\documentclass[letterpaper, oneside]{amsart}
\usepackage{layout}	

\usepackage[
]{geometry}

\usepackage[parfill]{parskip}   
\usepackage{amssymb}
\usepackage{amsthm}
\usepackage{amsmath}
\usepackage{enumitem}
\usepackage{mathtools}
\usepackage{hyperref}
\usepackage{colonequals}
\usepackage{mathrsfs}
\usepackage{color}
\usepackage{subcaption}
\usepackage{adjustbox}
\usepackage{ytableau}
\usepackage{bbm}

\usepackage{tikz}
\usetikzlibrary{cd}
\usetikzlibrary{decorations.markings}
\usetikzlibrary{arrows.meta}
\usetikzlibrary{patterns}

\tikzset{
    triple/.style args={[#1] in [#2] in [#3]}{
        #1,preaction={preaction={draw,#3},draw,#2}
    }
}      

\usepackage{tensor}

\theoremstyle{plain}
\newtheorem{thm}{Theorem}[section]
\newtheorem{lem}[thm]{Lemma}
\newtheorem{prop}[thm]{Proposition}
\newtheorem{cor}[thm]{Corollary}

\theoremstyle{definition}
\newtheorem{defn}[thm]{Definition}

\newtheorem{example}[thm]{Example}

\theoremstyle{remark}
\newtheorem*{rmk}{Remark}

\numberwithin{equation}{section}

\newcommand{\bA}{\mathbb{A}}

\newcommand{\bF}{\mathbb{F}}

\newcommand{\bN}{\mathbb{N}}
\newcommand{\bQ}{\mathbb{Q}}

\newcommand{\bZ}{\mathbb{Z}}

\newcommand{\bP}{\mathbb{P}}

\newcommand{\cB}{\mathcal{B}}
\newcommand{\cH}{\mathcal{H}}

\newcommand{\cP}{\mathcal{P}}

\newcommand{\fB}{\mathfrak{B}}
\newcommand{\fb}{\mathfrak{b}}

\newcommand{\fg}{\mathfrak{g}}

\newcommand{\ii}{\mathbbm{i}}
\newcommand{\vv}{\mathbbm{v}}
\newcommand{\ww}{\mathbbm{w}}

\newcommand{\Fq}{\mathbb{F}_q}
\newcommand{\Fqs}{\mathbb{F}_{q^2}}
\newcommand{\Fqbar}{\overline{\mathbb{F}_q}}

\newcommand{\br}[1]{\left\langle{#1}\right\rangle}

\newcommand{\qlbar}{{\overline{\mathbb{Q}_\ell}}}
\newcommand{\floor}[1]{\left\lfloor#1\right\rfloor}

\newcommand{\pch}[2]{\mathord{\downarrow}^{\left(#1\right)}_{\left(#2\right)}}

\newcommand{\tspq}[1]{\textup{\textbf{TSp}}_q\left(#1\right)}
\newcommand{\tspx}[1]{\textup{\textbf{TSp}}_x\left(#1\right)}
\newcommand{\chix}[1]{\mathcal{X}_x\left(#1\right)}

\newcommand{\spn}[1]{\textup{span}\left(#1\right)}

\DeclareMathOperator{\im}{\textup{im}} 

\DeclareMathOperator{\Res}{\textup{Res}}
\DeclareMathOperator{\tr}{\textup{tr}}

\DeclareMathOperator{\Ad}{\textup{Ad}}

\DeclareMathOperator{\Lie}{\textup{Lie}}

\DeclareMathOperator{\Spec}{\textup{Spec}}

\title{On the Betti numbers of Springer fibers for classical types}
\author{Dongkwan Kim}
\address{School of Mathematics\\
  University of Minnesota Twin Cities\\
  Minneapolis, MN 55455\\
  U.S.A.}
\email{kim00657@umn.edu}
\date{\today}							

\begin{document}
\begin{abstract} 
For a Weyl group $W$ of classical type, we present a formula to calculate the restriction of (graded) Springer representations of $W$ to a maximal parabolic subgroup $W'$ where the types of $W$ and $W'$ are in the same series. As a result, we obtain recursive formulas for the Betti numbers of Springer fibers for classical types.
\end{abstract}

\setcounter{tocdepth}{1}
\maketitle

\renewcommand\contentsname{}
\tableofcontents

\section{Introduction}
Let $G$ be a connected reductive group over $\Fq$ where $q$ is a power of some prime $p$ and let $W$ be its Weyl group. When $p$ and $q$ are sufficiently large, Springer \cite{spr76} defined the action of $W$ on the cohomology groups of the Springer fiber $\cB_N$ of any nilpotent element $N \in \Lie G$ defined over $\Fq$. Later, the condition on $p$ and $q$ is removed by Lusztig \cite{lus81} using the theory of perverse sheaves.

Now we suppose that $G$ is of classical type and split over $\Fq$. When $G$ is of type $A$, then its Springer representations are explicitly determined in terms of Green polynomials first appeared in Green's paper \cite{gre55}. Similarly, when $G$ is of type $B, C,$ or $D$ and $p>2$, then Srinivasan \cite{sri77} proved that its Springer representations are also described by some polynomials, which are now called Green functions. In other words, when $N \in \Lie G$ is a nilpotent element fixed by the standard Frobenius morphism $F$, then the eigenvalues of $F^*$ on each $H^i(\cB_N)$ is given by $\pm q^j$ for some $j \in \bN$, see \cite[Theorem 1]{sri77}.

The aim of this paper is to obtain some restriction formulas about these Green functions of classical types by exploiting the method of \cite{sri77}. To be precise, suppose that $W'$ is a maximal parabolic subgroup of $W$ such that the types of $W$ and $W'$ are in the same series. Then for any $N \in \Lie G$ and for each $i \in \bZ$, we express $\Res^W_{W'} H^i(\cB_N)$ in terms of Springer representations corresponding to $W'$. As a result, we also obtain inductive formulas for the Betti numbers of Springer fibers.

Note that there exists an algorithm developed by Shoji \cite{sho83} and Lusztig \cite{lus86} which calculates  the Green functions using the orthogonality of such functions and the theory of character sheaves. As our formula only gives information about the restriction of Springer representations to a fixed maximal parabolic subgroup, our result is weaker than their algorithm. However, we believe that our method is much faster than their algorithm if one only wants to calculate the Betti numbers of Springer fibers.

Finally, we remark that this paper supersedes one of the author's previous papers \cite{kim:euler}. In particular, if we forget the cohomological gradings on the main theorem, then we recover the result of \cite{kim:euler}.

This paper is organized as follows: in Section \ref{sec:setup} we review some basic notations and definitions which will be used throughout this paper; in Section \ref{sec:adapted} we define a notion of a standard model that is essential for most of the calculation in this paper; in Section \ref{sec:mainthm} we state our main theorem; in Section \ref{sec:conj} we recall the parametrization of rational nilpotent orbits in the Lie algebra of $G$; in Section \ref{sec:partial} partial Springer resolution studied by Borho and MacPherson \cite{bm83} is revisited; in Section \ref{sec:rect} we prove the main theorem in some special cases when all the Jordan blocks of a nilpotent element are the same size; in Section \ref{sec:gen} we generalize the result in the previous section and complete the proof of the main theorem.

\section{Setup}\label{sec:setup}
 Let $p>2$ be a prime number and $q$ be some power of $p$. We assume that $p$ and $q$ are sufficiently large, and that $q \equiv 1 \pmod{8}$ so that $2, -1 \in (\bF_q^\times)^2$. (However, the result of this paper is still valid for any prime $p>2$ and arbitrary $q$.) We fix the square root of 2 and -1, denoted $\sqrt{2}, \ii \in \Fq^\times, $ respectively.
 
For a variety $X$ defined over $\Fqbar$, we say that $X$ is defined over $\Fq$ if there exists $X_0$ over $\Fq$ such that $X=X_0 \times_{\Spec \Fq} \Spec \Fqbar.$ If so, we define $F$ to be the geometric Frobenius morphism $F: X\rightarrow X$ with respect to the corresponding $\Fq$-structure on $X$. Also write $X^F$ to be the set of (closed) points in $X$ fixed by $F$.
 
For a variety $X$, let $\qlbar_X$ be the constant $\qlbar$-adic sheaf on $X$ where $\ell$ is a prime different from the characteristic of $X$. For $i \in \bZ$, define $H^i(X) \colonequals H^i(X, \qlbar_X)$ to be the $i$-th $\ell$-adic cohomology group of $X$ and let $H^*(X) \colonequals \bigoplus_{i \in \bZ}(-1)^i H^i(X)$ be their alternating sum (as a virtual vector space). If $X$ is defined over $\Fq$, then there exists a canonical endomorphism $F^*$ on each $H^i(X)$ and also on $H^*(X)$

Let $G$ be either $SO_{2n+1}, Sp_{2n},$ or $SO_{2n}$ which is split over $\Fq$. Write $\fg$ to be the Lie algebra of $G$. Then any element $g\in G$ acts on $\fg$ by adjoint action, denoted by $\Ad_g: \fg \rightarrow \fg$. Also, we let $(V, \br{\ ,\ })$ be the defining representation of $G$ where $\br{\ , \ }: V \times V \rightarrow \Fqbar$ is an $F$-stable bilinear form respected by $G$. It is symmetric if $G=SO_{2n+1}$ or $G=SO_{2n}$, and symplectic if $G=Sp_{2n}$.

Let $W$ be the Weyl group of $G$ with a set $S$ of simple reflections. (Note that $F$ acts trivially on $W$ by assumption.) If $G$ is either $SO_{2n+1}$ or $Sp_{2n}$ for some $n \geq 1$, then we set $S=\{s_1, s_2, \ldots, s_n\}$ such that $(s_1s_2)^4=(s_is_{i+1})^3=id$ for any $2\leq i \leq n-1$, and let $W'$ be the maximal parabolic subgroup of $W$ generated by $s_1, s_2, \ldots, s_{n-1}$. If $G$ is $SO_{2n}$ for some $n \geq 2$, then we set $S=\{s_1, s_2, \ldots, s_n\}$ such that $(s_1s_3)^3=(s_2s_3)^3=(s_1s_2)^2=(s_is_{i+1})^3=id$ for any $3 \leq i \leq n-1$. If $n\geq 3$, then let $W'$ be the maximal parabolic subgroup of $W$ generated by $s_1, s_2, \ldots, s_{n-1}$. If $n=2$, then set $W'\colonequals\{id\}.$

Let $\cB$ be the flag variety of $G$ consisting of its Borel subgroups. For a nilpotent element $N \in \fg$, we define the Springer fiber of $N$ to be the variety $\cB_N \colonequals \{ B \in \cB \mid N \in \Lie B\}.$ Then each $H^i(\cB_N)$ is equipped with the Springer action of $W$ originally defined by \cite{spr76}. Here we adopt the convention of \cite{lus81} that $H^0(\cB_N)$ gives the trivial representation of $W$ for some/any nilpotent $N$. (It differs from \cite{spr76} by the sign character of $W$.) If $N$ is $F$-stable, then the endomorphism $F^*$ on each $H^i(\cB_N)$ commutes with the action of $W$. Furthermore, by \cite[3.9]{sho83} $H^i(\cB_N)=0$ if $i$ is odd and the eigenvalues of $F^*$ on $H^{2i}(\cB_N)$ are $\pm q^i$.

For a nilpotent $N \in \fg$, let $Z_G(N) \subset G$ be the stabilizer of $N$ under adjoint action of $G$. Then $Z_G(N)$ acts naturally on each $H^i(\cB_N)$, which descends to the action of its component group $A(N) \colonequals Z_G(N)/Z_G(N)^0$. It is known that the action of $A(N)$ and $W$ on $H^i(\cB_N)$ commute. Also, $F$ acts trivially on $A(N)$ \cite[3.1]{sho83} and thus $F^*$ and the action of $A(N)$ on $H^i(\cB_N)$ also commute.

For a partition $\lambda \vdash n$, we write either $\lambda=(\lambda_1, \lambda_2, \ldots)$ or $\lambda=(1^{m_1}2^{m_2}\cdots)$ to describe its parts. For example, we have $(6,6,4)=(4^16^2) \vdash 16$. Also, for $a_1, a_2, \ldots, a_k, b_1, b_2, \ldots, b_k \in \bN$ such that $\{a_1, a_2, \ldots, a_k\} \subset \{\lambda_1, \lambda_2, \ldots\}$ as a multiset, we define $\lambda\pch{a_1, a_2, \ldots, a_k}{b_1, b_2, \ldots, b_k}$ to be the partition whose parts are obtained from $\{\lambda_1, \lambda_2, \ldots\}$ by replacing $a_1, a_2, \ldots, a_k$ with $b_1, b_2, \ldots, b_k$ and rearranging it if necessary. For example, we have $(6,4,4,4, 3,2)\pch{4,4}{2,1} = (6,4,3,2,2,1)$.

For a nilpotent element $N \in \fg$, we define its Jordan type to be the partition whose parts are the sizes of Jordan blocks of $N$ considered as an endomorphism on $V$. If the Jordan type of $N$ is $\lambda$, we also say that $N$ is of Jordan type $\lambda$. If $G$ is $SO_{2n+1}$, then the nilpotent orbits in $\fg$ are parametrized by the set of partitions
$$\{ \lambda=(1^{m_1}2^{m_2} \cdots) \vdash 2n+1 \mid m_{2i} \textup{ is even for } i\in \bZ_{>0}\}$$
such that any element in the nilpotent orbit parametrized by $\lambda$ is of Jordan type $\lambda$. Similarly if $G$ is $Sp_{2n}$, then the nilpotent orbits in $\fg$ are parametrized by
$$\{ \lambda=(1^{m_1}2^{m_2} \cdots) \vdash 2n \mid m_{2i-1} \textup{ is even for } i\in \bZ_{>0}\}.$$
If $G$ is $SO_{2n}$, then there exists such a bijection between the nilpotent orbits in $\fg$ and
$$\{ \lambda=(1^{m_1}2^{m_2} \cdots) \vdash 2n \mid m_{2i} \textup{ is even for } i\in \bZ_{>0}\}$$
except that if $\lambda = (1^{m_1}2^{m_2} \cdots)$ satisfies $m_{2i-1}=0$ for all $i \in \bZ_{>0}$ then it corresponds to two nilpotent orbits in $\fg$. Such $\lambda$ is called a very even partition. In this case we use $\lambda+$ and $\lambda-$ to parametrize two such orbits (if necessary).

Note that $A(N)$ is an $\bF_2$-vector space for any nilpotent $N\in \fg$. If $N$ and $N'$ are in the same $G$-orbit, i.e. if there exists $g\in G$ such that $\Ad_g N=N'$, then it induces an isomorphism $\Ad_g: A(N)\rightarrow A(N')$. This does not depend on the choice of $g$ because for any $h \in Z_G(N)$, the isomorphism $\Ad_h : A(N) \simeq A(N)$ is equal to the identity map (since $A(N)$ is abelian). Thus for each nilpotent orbit in $\fg$ we may attach an abstract group which is canonically isomorphic to $A(N)$ for any element $N$ in this orbit. If such an orbit is parametrized by $\lambda$, then we denote this abstract group by $A(\lambda)$.

If $G=SO_{2n}$ and $\lambda$ is very even, we define $A(\lambda+)$ and $A(\lambda-)$ similarly. However, we may further identify $A(\lambda+)$ and $A(\lambda-)$; if the orbits of $N$ and $N'$ are parametrized by $\lambda+$ and $\lambda-$, respectively, then there exists $g \in O_{2n} - SO_{2n}$ such that $\Ad_gN = N'$. Then it descends to the isomorphism $\Ad_g : A(N) \rightarrow A(N')$, which is also canonical since $Z_{O_{2n}}(N)/Z_{O_{2n}}(N)^0$ is abelian. (Note that this group naturally contains $A(N)$ as a subgroup of index 2.) Thus we may define $A(\lambda)$ to be the abstract group canonically isomorphic to $A(N)$ for any $N \in \Lie SO_{2n}$ of Jordan type $\lambda$ even when $\lambda$ is very even.

Recall that there are actions of $W$ and $A(N)$ on $H^i(\cB_N)$ that commute with each other. Thus if $N$ is of Jordan type $\lambda$, we may regard $H^i(\cB_N)$ as a $W\times A(\lambda)$-module. Now for any $z \in A(\lambda)$ we define
$$\tspx{\lambda, z}\colon W \rightarrow \qlbar[x^{\pm 1/2}]: w \mapsto \sum_{i \in \bZ} (-1)^i \tr (wz, H^i(\cB_N)) x^{i/2}$$
where $x$ is an indeterminate. By \cite[Theorem 2]{sri77}, in fact this is a $\bQ[x]$-valued virtual character of $W$. When $G=SO_{2n}$ and $\lambda$ is very even, we define $\tspx{\lambda+,z}$ and $\tspx{\lambda-, z}$ similarly for any $z\in A(\lambda)$, and let 
$$\tspx{\lambda, z} \colonequals \frac{1}{2}(\tspx{\lambda+,z}+\tspx{\lambda-, z}).$$
Finally, we define $\tspq{\lambda, z} \colonequals \tspx{\lambda, z}|_{x=q}$.

\section{Standard model}\label{sec:adapted}
\subsection{Orthogonal decomposition with respect to $N$} \label{subsec:decomp} Suppose that $N \in \fg$ is a nilpotent element of Jordan type $(1^{m_1}2^{m_2} \cdots)$. Then there exists an orthogonal decomposition
$$V = V_1 \oplus V_2 \oplus \cdots$$
where each $V_i$ is $N$-stable and $N|_{V_i}$ is of Jordan type $(i^{m_i})$. Note that this decomposition is not canonical. For any nilpotent $N \in \fg$, we usually assume that such a decomposition of $V$ is also given. In such a case we define $\ker_i N \colonequals \ker N \cap V_i$. Likewise, let $V_{>i} \colonequals \bigoplus_{j>i} V_j$ and $\ker_{>i} N \colonequals \ker N \cap V_{>i}$, and define $V_{\geq i}, V_{< i}, V_{\leq i}, \ker_{\geq i} N, \ker_{< i} N,$ and  $\ker_{\leq i} N$ in an analogous manner.

\subsection{A basis adapted to $N$} \label{subsec:adapted}
Let $N \in \fg$ be a nilpotent element of Jordan type $\lambda$. We define a certain basis of $V$ as follows. First assume that $\lambda = (i^m)$ for some $i,m \in \bZ_{>0}$. Then we choose a basis $\{\vv_{s,t} \mid 1\leq s \leq m, 1 \leq t \leq i\}$ which satisfies 
\begin{gather*}
N \vv_{s,t}=\vv_{s,t+1} \textup{ if } 1\leq t<i, \quad N \vv_{s,i} = 0,
\\\br{\vv_{s,t}, \vv_{s',t'}}=0 \textup{ unless } s+s' = m+1 \textup{ and } t+t' = i+1,
\end{gather*}
and the following holds.
\begin{enumerate}[label=$\bullet$]
\item If $G=Sp_{mi}$ and $i$ is odd or $G=SO_{mi}$ and $i$ is even, then 
$$
\br{\vv_{s,t}, \vv_{m+1-s, i+1-t}}= \left\{
\begin{aligned}
&(-1)^{t+1} & \textup{ if } 1 \leq s\leq m/2
\\&(-1)^t & \textup{ if } m/2 < s \leq m
\end{aligned}
\right.
$$
\item Otherwise
$$\br{\vv_{s,t}, \vv_{m+1-s, i+1-t}}= (-1)^{t+1}.$$
\end{enumerate}
Note that it is always possible to choose such a basis.

In general, for $N \in \fg$ of Jordan type $(1^{m_1} 2^{m_2} 3^{m_3} \cdots)$ we let $V=V_1 \oplus V_2 \oplus \cdots$ be the decomposition defined in \ref{subsec:decomp}. Then we choose a basis $\{\vv_{s,t}^i\}_{s,t}$ of each $V_i$ as above with respect to $N|_{V_i}$. Then their union $\{\vv_{s,t}^i\}_{i,s,t}$ is a basis of $V$.

\begin{defn} Let $N \in \fg$ is a nilpotent element. If a basis $\fB$ of $V$ is (up to permutation) the same as the one defined above with respect to $N$, then we say that $\fB$ is adapted to $N$.
\end{defn}
Note that if $\fB$ is adapted to $N$, then $N\fB \subset \fB\cup \{0\}$.

\subsection{Description of $A(N)$ and a standard model} Here we describe a representative of each element of $A(N)$ in $Z_G(N)$. To this end first we set $\tilde{G} = O_{2n+1}$ (resp. $O_{2n}$) if $G=SO_{2n+1}$ (resp. $SO_{2n}$), and $\tilde{G}=G$ if $G=Sp_{2n}$. Let $\tilde{A}(N)$ and $Z_{\tilde{G}}(N)$ be the groups analogously defined to $A(N)$ and $Z_{G}(N)$ where $G$ is replaced by $\tilde{G}$. Note that when $G\subsetneq \tilde{G}$, $A(N)$ is naturally an index 2 subgroup of $\tilde{A}(N)$.

Assume that $N \in \fg$ is of Jordan type $(i^m)$ for some $i, m \in \bZ_{>0}$. If $G=Sp_{mi}$ and $i$ is odd, or $G=SO_{mi}$ and $i$ is even, then $\tilde{A}(N)$ is trivial. Otherwise, let $\{\vv_{s,t}\}_{s,t}$ be the basis of $V$ defined in \ref{subsec:adapted} and consider the following two cases.
\begin{enumerate}[label=$\bullet$]
\item Suppose that $m$ is odd. Then define $\tilde{z}: V \rightarrow V$ to be the linear map characterized by
$$\tilde{z}(\vv_{s,t}) = \vv_{s,t} \textup{ unless } s \neq \frac{m+1}{2}, \quad  \tilde{z}(\vv_{(m+1)/2,t}) = -\vv_{(m+1)/2,t}.$$
\item Suppose that $m$ is even. Then define $\tilde{z}: V \rightarrow V$ to be the linear map characterized by
$$\tilde{z}(\vv_{s,t}) = \vv_{s,t} \textup{ unless } s \neq \frac{m}{2} \textup{ or } \frac{m+2}{2}, \quad \tilde{z}(\vv_{m/2,t}) = \vv_{(m+2)/2,t}, \quad \tilde{z}(\vv_{(m+2)/2,t}) = \vv_{m/2,t}.$$
\end{enumerate}
In each case, it is easy to show that $\tilde{z} \in Z_{\tilde{G}}(N)$ and the image of $\tilde{z}$ in $\tilde{A}(N)$ equals the unique nontrivial element of $\tilde{A}(N)$.

In general, for $N \in \fg$ of Jordan type $\lambda=(1^{m_1} 2^{m_2} 3^{m_3} \cdots)$ we consider the decomposition
$V=V_1 \oplus V_2 \oplus \cdots$ as in \ref{subsec:decomp}. When $G=Sp_{2n}$, for each even $i$ define $\tilde{z}_i : V_i \rightarrow V_i$ to be the linear map defined above and extend trivially to $V$. When $G=SO_{2n+1}$ or $G=SO_{2n}$, for each odd $i$ we define $\tilde{z}_i :V_i \rightarrow V_i$ similarly. Then it is also clear that $\tilde{z}_i \in Z_{\tilde{G}}(N)$ and the image of $\{\tilde{z}_i\}_{i}$ generates $\tilde{A}(N)$. 

Recall the definition of $A(\lambda)$, an abstract group canonically isomorphic to $A(N)$. We define $\tilde{A}(\lambda)$ analogously to $A(\lambda)$. Let $z_i \in \tilde{A}(\lambda)$ be the image of $\tilde{z}_i \in Z_{\tilde{G}}(N)$ under the canonical homomorphism $Z_{\tilde{G}}(N) \twoheadrightarrow \tilde{A}(N) \simeq \tilde{A}(\lambda)$. Also when $G=Sp_{2n}$ (resp. $G=SO_{2n+1}$ or $G=SO_{2n}$), if $m_i = 0$ or $i$ is not even (resp. not odd) then put $z_i = id$. If $G=Sp_{2n}$ then we have
$$A(\lambda) = \tilde{A}(\lambda) =\left\{{\textstyle \prod_{i \textup{ even}, m_i >0} }\ z_i^{a_i} \mid a_i \in \{0,1\}\right\}.$$
Similarly if $G=SO_{2n+1}$ or $G=SO_{2n}$ then
\begin{align*}
\tilde{A}(\lambda) &=\{{\textstyle \prod_{i \textup{ odd}, m_i >0} }\ z_i^{a_i} \mid a_i \in \{0,1\}\}
\\A(\lambda) &= \{{\textstyle \prod_{i \textup{ odd}, m_i >0} }\ z_i^{a_i} \in\tilde{A}(\lambda) \mid \sum a_i \in 2\bN\}.
\end{align*}

We define the notion of a standard model. Later, most calculation in this paper will begin with a proper choice of some standard model.
\begin{defn} For $N \in \fg$ of Jordan type of $\lambda$, the triple $(N, \{\vv^i_{s,t}\}_{i,s,t}, \{\tilde{z}_i\}_i)$ defined above is called a standard model (or a standard model for $\lambda$ if we need to specify the Jordan type $\lambda$ of $N$).
\end{defn}
Note that $ \{\tilde{z}_i\}_i$ is a subset of $Z_{\tilde{G}}(N)$, whose image generates $\tilde{A}(N)$.

\section{Main theorem}\label{sec:mainthm}
We are ready to state the main theorem of this paper.
\begin{thm} \label{thm:main}Let $\lambda=(1^{m_1}2^{m_2} \cdots)$ be the Jordan type of some nilpotent element in $\fg$ and let $z=\prod_j z_j^{a_j}\in A(\lambda)$ for some $a_j \in \{0,1\}$. Write $m_{>i} \colonequals \sum_{j > i} m_j$.
\begin{enumerate}[label=\textup{(\arabic*)}]
\item If $G=SO_{2n+1}$ or $G=SO_{2n}$ for some $n \geq 2$, then $\Res^W_{W'}\tspx{\lambda, z}$ is equal to
\begin{align*}
& \sum_{i \textup{ even}}  x^{m_{>i}}\frac{x^{m_i}-1}{x-1}  \tspx{\lambda\pch{i,i}{i-1,i-1}, z}
\\& +\sum_{\substack{i \textup{ odd}\\m_i \textup{ odd}}} x^{m_{>i}} \left[
\begin{aligned}
&\frac{x^{{m_i}-1}-1}{x-1}\tspx{\lambda\pch{i,i}{i-1,i-1}, z}
\\&+\frac{x^{{m_i}-1}+x^{(m_i-1)/2}}{2}\tspx{\lambda\pch{i}{i-2},z (z_iz_{i-2})^{a_i}}
\\&+\frac{x^{{m_i}-1}-x^{(m_i-1)/2}}{2}\tspx{\lambda\pch{i}{i-2},z (z_i z_{i-2})^{a_i+1}}
\end{aligned}
\right]
\\&+\sum_{\substack{i \textup{ odd}\\m_i \textup{ even}}} x^{m_{>i}} \left[
\begin{aligned}
& \left(\frac{x^{m_i-1}-1}{x-1}+(-1)^{a_i}x^{m_i/2-1}\right)\tspx{\lambda\pch{i,i}{i-1,i-1}, z}
\\&+\frac{x^{{m_i}-1}-(-1)^{a_i}x^{m_i/2-1}}{2}\tspx{\lambda\pch{i}{i-2},z (z_iz_{i-2})^{a_i}}
\\&+\frac{x^{{m_i}-1}-(-1)^{a_i}x^{m_i/2-1}}{2}\tspx{\lambda\pch{i}{i-2},z (z_i z_{i-2})^{a_i+1}}
\end{aligned}
\right]
\end{align*}
as a character of $W'$. Here, in the expression above $z = \prod_j z_j^{a_j}$ is understood as an element of $A(\lambda\pch{i,i}{i-1,i-1})$ or $A(\lambda\pch{i}{i-2})$. Also, we put $\tspx{(1,1),id} = 1$ and $\tspx{\lambda\pch{1}{-1},-}=0$.
\item If $G=Sp_{2n}$ for some $n \geq 2$, then $\Res^W_{W'}\tspx{\lambda, z}$ is equal to
\begin{align*}
& \sum_{i \textup{ odd}}  x^{m_{>i}}\frac{x^{m_i}-1}{x-1}  \tspx{\lambda\pch{i,i}{i-1,i-1}, z}
\\& +\sum_{\substack{i \textup{ even}\\m_i \textup{ odd}}} x^{m_{>i}} \left[
\begin{aligned}
&\frac{x^{{m_i}-1}-1}{x-1}\tspx{\lambda\pch{i,i}{i-1,i-1}, z}
\\&+\frac{x^{{m_i}-1}+x^{(m_i-1)/2}}{2}\tspx{\lambda\pch{i}{i-2},z (z_iz_{i-2})^{a_i}}
\\&+\frac{x^{{m_i}-1}-x^{(m_i-1)/2}}{2}\tspx{\lambda\pch{i}{i-2},z (z_i z_{i-2})^{a_i+1}}
\end{aligned}
\right]
\\&+\sum_{\substack{i \textup{ even}\\m_i \textup{ even}}} x^{m_{>i}} \left[
\begin{aligned}
& \left(\frac{x^{m_i-1}-1}{x-1}+(-1)^{a_i}x^{m_i/2-1}\right)\tspx{\lambda\pch{i,i}{i-1,i-1}, z}
\\&+\frac{x^{{m_i}-1}-(-1)^{a_i}x^{m_i/2-1}}{2}\tspx{\lambda\pch{i}{i-2},z (z_iz_{i-2})^{a_i}}
\\&+\frac{x^{{m_i}-1}-(-1)^{a_i}x^{m_i/2-1}}{2}\tspx{\lambda\pch{i}{i-2},z (z_i z_{i-2})^{a_i+1}}
\end{aligned}
\right].
\end{align*}
as a character of $W'$. Here, in the expression above $z = \prod_j z_j^{a_j}$ is understood as an element of $A(\lambda\pch{i,i}{i-1,i-1})$ or $A(\lambda\pch{i}{i-2})$. 
\end{enumerate}
\end{thm}
\begin{rmk} A similar formula holds for $G=GL_n$; see \cite[Remark 2.4]{hs77} or \cite[Theorem 4.4]{kim:euler}.
\end{rmk}

\begin{example} For brevity let us write $\chix{\lambda,z} \colonequals \tspx{\lambda,z}(id)$. Here we calculate some $\chix{\lambda,z}$ when $G=Sp_{2n}, n\in \{2, 3\}$ from the initial condition $\chix{(2),id}=\chix{(2),z_2}= 1$ and $\chix{(1,1),id}= x + 1$. First when $G=Sp_{4}$, we have
\begin{align*}
\\\chix{(4),id}&= \chix{(2),id} = 1\allowdisplaybreaks
\\\chix{(4),z_4}&= \chix{(2),z_2} = 1\allowdisplaybreaks
\\\chix{(2,2),id}&= \frac{x-1}{2}\chix{(2),z_2}+\frac{x-1}{2}\chix{(2),id}+(2)\chix{(1,1),id} = 3x + 1\allowdisplaybreaks
\\\chix{(2,2),z_2}&= \frac{x+1}{2}\chix{(2),z_2}+\frac{x+1}{2}\chix{(2),id} = x + 1\allowdisplaybreaks
\\\chix{(2,1,1),id}&= (x^2 + x)\chix{(2),id}+\chix{(1,1),id} = x^2 + 2x + 1\allowdisplaybreaks
\\\chix{(2,1,1),z_2}&= (x^2 + x)\chix{(2),z_2}+\chix{(1,1),id} = x^2 + 2x + 1\allowdisplaybreaks
\\\chix{(1,1,1,1),id}&= (x^3 + x^2 + x + 1)\chix{(1,1),id} = x^4 + 2x^3 + 2x^2 + 2x + 1\allowdisplaybreaks
\end{align*}

Similarly, when $G=Sp_6$ we have
\begin{align*}
\\\chix{(6),id}&= \chix{(4),id} = 1\allowdisplaybreaks
\\\chix{(6),z_6}&= \chix{(4),z_4} = 1\allowdisplaybreaks
\\\chix{(4,2),id}&= x\chix{(4),id}+\chix{(2,2),id} = 4x + 1\allowdisplaybreaks
\\\chix{(4,2),z_2}&= x\chix{(4),id}+\chix{(2,2),z_2} = 2x + 1\allowdisplaybreaks
\\\chix{(4,2),z_4}&= x\chix{(4),z_4}+\chix{(2,2),z_2} = 2x + 1\allowdisplaybreaks
\\\chix{(4,2),z_2z_4}&= x\chix{(4),z_4}+\chix{(2,2),id} = 4x + 1\allowdisplaybreaks
\\\chix{(4,1,1),id}&= (x^2 + x)\chix{(4),id}+\chix{(2,1,1),id} = 2x^2 + 3x + 1\allowdisplaybreaks
\\\chix{(4,1,1),z_4}&= (x^2 + x)\chix{(4),z_4}+\chix{(2,1,1),z_2} = 2x^2 + 3x + 1\allowdisplaybreaks
\\\chix{(3,3),id}&= (x + 1)\chix{(2,2),id} = 3x^2 + 4x + 1\allowdisplaybreaks
\\\chix{(2,2,2),id}&= \frac{x^2-x}{2}\chix{(2,2),z_2}+\frac{x^2+x}{2}\chix{(2,2),id}+(x + 1)\chix{(2,1,1),id} 
\\&= 3x^3 + 5x^2 + 3x + 1\allowdisplaybreaks
\\\chix{(2,2,2),z_2}&= \frac{x^2-x}{2}\chix{(2,2),z_2}+\frac{x^2+x}{2}\chix{(2,2),id}+(x + 1)\chix{(2,1,1),z_2}
\\& = 3x^3 + 5x^2 + 3x + 1\allowdisplaybreaks
\\\chix{(2,2,1,1),id}&= (x^3 + x^2)\chix{(2,2),id}+\frac{x-1}{2}\chix{(2,1,1),z_2}+\frac{x-1}{2}\chix{(2,1,1),id}
\\&\quad +2\chix{(1,1,1,1),id}  = 5x^4 + 9x^3 + 6x^2 + 3x + 1\allowdisplaybreaks
\\\chix{(2,2,1,1),z_2}&= (x^3 + x^2)\chix{(2,2),z_2}+\frac{x+1}{2}\chix{(2,1,1),z_2}+\frac{x+1}{2}\chix{(2,1,1),id}
\\& = x^4 + 3x^3 + 4x^2 + 3x + 1\allowdisplaybreaks
\\\chix{(2,1,1,1,1),id}&= (x^4 + x^3 + x^2 + x)\chix{(2,1,1),id}+\chix{(1,1,1,1),id} 
\\&= x^6 + 3x^5 + 5x^4 + 6x^3 + 5x^2 + 3x + 1\allowdisplaybreaks
\\\chix{(2,1,1,1,1),z_2}&= (x^4 + x^3 + x^2 + x)\chix{(2,1,1),z_2}+\chix{(1,1,1,1),id} 
\\&= x^6 + 3x^5 + 5x^4 + 6x^3 + 5x^2 + 3x + 1\allowdisplaybreaks
\\\chix{(1,1,1,1,1,1),id}&= (x^5 + x^4 + x^3 + x^2 + x + 1)\chix{(1,1,1,1),id} 
\\&= x^9 + 3x^8 + 5x^7 + 7x^6 + 8x^5 + 8x^4 + 7x^3 + 5x^2 + 3x + 1\allowdisplaybreaks
\end{align*}
\end{example}

The rest of this paper is devoted to the proof of this main theorem. From now on, we assume that $G=Sp_{2n}$ for some $n \geq 2$ and give a proof in this case. However, the cases when $G=SO_{2n+1}$ and $G=SO_{2n}$ can be proved in almost the same manner. Later we give a remark for the proof in such cases at the end of each following section.

\section{Rational orbits of nilpotent elements}\label{sec:conj}
\subsection{Quadratic form $Q$ on $\ker N$} \label{subsec:Q}
We assume that $N\in \fg$ is a nilpotent element of Jordan type $(i^m)$ for some $i,m \in \bZ_{>0}$. Then there exists a canonical quadratic form $Q: \ker N \rightarrow \Fqbar$ as follows. For $v \in \ker N$, choose any $v' \in V$ such that $N^{i-1}v' = v$. (Such $v'$ always exist since $\ker N = \im N^{i-1}$.) Then it is easy to show that $\br{v', v}$ does not depend on the choice of $v'$ but only on $v$. We define $Q:  \ker N \rightarrow \Fqbar: v \mapsto \br{v', v}$. Then $Q$ is easily seen to be a quadratic form on $\ker N$. (ref. \cite[p.1246]{sri77}, \cite[2.1]{sho83}, \cite[2.3]{vl89})

If $i$ is odd, then direct calculation shows that $Q=0$. Otherwise, $Q$ is non-degenerate. We recall the basis $\{\vv_{s,t} \mid 1 \leq s \leq m, 1\leq t \leq i\}$ defined in \ref{subsec:adapted}. Then clearly $\{\vv_{s,i}\}_{s}$ is a basis of $\ker N$. Now we identify $\ker N$ with $\Fqbar^m$ where $v =x_1 \vv_{1,i}+x_2 \vv_{2,i} + \cdots$ is sent to $(x_1, x_2, \ldots)$. Then $Q$ is given by the equation
\begin{align*}
&2x_1 x_m +2 x_2 x_{m-1} + \cdots + 2x_{\frac{m-1}{2}}x_{\frac{m+3}{2}}+(x_{\frac{m+1}{2}})^2 &\textup{if $m$ is odd and}
\\&2x_1 x_m +2 x_2 x_{m-1} + \cdots + 2x_{\frac{m}{2}-1}x_{\frac{m}{2}+2}+2x_{\frac{m}{2}}x_{\frac{m}{2}+1}
&\textup{if $m$ is even.}
\end{align*}

In general, for a nilpotent element $N \in \fg$ of Jordan type $(1^{m_1}2^{m_2} \cdots)$ consider the decomposition $V= V_1 \oplus V_2 \oplus \cdots$ defined in \ref{subsec:decomp}. Then similarly there exists a canonical quadratic form on each $\ker_i N$, which we denote by $Q^i$.

\subsection{Quadratic form $Q_{F'}$ on $(\ker N)^{F'}$.}  \label{subsec:QF} Before we proceed, let us recall the classification of quadratic forms over $\Fq$. Let $\tilde{V}=\Fq^n$ and $\tilde{Q}$ be a quadratic form $\tilde{Q}: \tilde{V} \rightarrow \Fq$. We say that $\tilde{Q}$ is split if it is isometric over $\Fq$ to the one given by the equation
\begin{align*}
&2(x_1x_2+x_3x_4+\cdots + x_{n-1}x_n) &\textup{ if $n$ is even, and}
\\&2(x_1x_2+x_3x_4+\cdots + x_{n-2}x_{n-1})+x_n^2 & \textup{ if $n$ is odd}
\end{align*}
where $(x_1, x_2, \ldots, x_n)$ are $\Fq$-coordinates of $\tilde{V}$. Otherwise it is called non-split. Then it is known that $\tilde{Q}$ is split if and only if $\det \tilde{Q}$ is a quadratic residue modulo $q$. (Recall that $-1$ is a quadratic residue modulo $q$ by our assumption.)

Now let $F'=\Ad_g \circ F$ for some $g\in G$, thus in particular $F'$ respects $\br{\ , \ }$. Let $N \in \fg$ be an $F'$-stable nilpotent element of Jordan type $(i^m)$ for some $i, m \in \bZ_{>0}$ and suppose that $i$ is even. Then the quadratic form $Q: \ker N \rightarrow \Fqbar$ is non-degenerate and $F'$-stable, thus  $Q$ restricts to a well-defined quadratic form $Q_{F'}: (\ker N)^{F'} \rightarrow \Fq$. We recall the basis $\{\vv_{s,t}\}_{s,t}$ of $V$ defined in \ref{subsec:adapted}. If $F'(\vv_{s,i}) = \vv_{s,i}$ for each $1\leq s \leq m$, then direct calculation shows that $\det Q_{F'}$
is a quadratic residue modulo $q$. Thus in particular $Q_{F'}$ is split.

In general, for an $F'$-stable nilpotent element $N \in \fg$ of Jordan type $(1^{m_1}2^{m_2} \cdots)$ consider the decomposition $V= V_1 \oplus V_2 \oplus \cdots$ defined in \ref{subsec:decomp}. Then each $Q^i$ restricts to a quadratic from on $(\ker_i N)^{F'}$, which we denote by $Q^i_{F'}$. Similarly, if $i$ is even and $F'(\vv^i_{s,i}) = \vv^i_{s,i}$ for any $1\leq s\leq m_i$ then $Q^i_{F'}$ is split.

\subsection{Split nilpotent element}
Here we recall the notion of a split nilpotent element. Again $F' = \Ad_g\circ F$ for some $g\in G$.
\begin{defn}[{\cite{sho83, bs84}}] Let $N \in \fg$ be an $F'$-stable nilpotent element. Then we say that $N$ is split with respect to $F'$ if $F'$ acts trivially on the set of irreducible components of $\cB_N$. In particular if $F'=F$, then we say that $N$ is split. (It was called distinguished in \cite{sho83}.)
\end{defn}
There is a strong connection between splitness of $N$ and corresponding quadratic forms $Q^i_{F'}$ as described in the following proposition.
\begin{prop} Let $N \in \fg$ be an $F'$-stable nilpotent element of Jordan type $(1^{m_1}2^{m_2} \cdots)$. Then $N$ is split with respect to $F'$ if and only if the following two statements hold:
\begin{enumerate}[label=$\bullet$]
\item For even $i$ such that $m_i$ is even, any $Q^i_{F'}$ is split.
\item For even $i$ such that $m_i$ is odd, either all $Q^i_{F'}$ are split or all $Q^i_{F'}$ are non-split.
\end{enumerate}
\end{prop}
\begin{proof} It is exactly \cite[Proposition 3.3(i)]{sho83}.
\end{proof}

\subsection{Rational nilpotent orbits}
Here we describe rational nilpotent orbits in $\fg^{F'}$ where $F' = \Ad_g\circ F$ for some $g\in G$. First of all, we define the notion of a rational standard model. 
\begin{defn} Assume that $N$ is of Jordan type $\lambda = (1^{m_1}2^{m_2} \cdots)$. Then $(N, \{\vv^i_{s,t}\}_{i,s,t}, \{\tilde{z}_i\}_i, F')$ is called a rational standard model (for $\lambda$) if $(N, \{\vv^i_{s,t}\}_{i,s,t}, \{\tilde{z}_i\}_i)$ is a standard model (for $\lambda$) and $F'(a\vv^i_{s,t}) = a^q\vv^i_{s,t}$ for any $i \in \bZ_{>0}, 1\leq s\leq i, 1\leq t\leq m_i$, and $a \in \Fqbar$. (In particular, $F'(\vv^i_{s,t})=\vv^i_{s,t}, F'(N)=N$, and $\tilde{z}_i\circ F'=F'\tilde\circ{z}_i$.)
\end{defn}

Now we describe rational nilpotent orbits in $\fg^{F'}$. If $N \in \fg$ is $F'$-stable, then $(G\cdot N)^{F'}$ splits into rational orbits, i.e. orbits of $G^{F'}$. Suppose that $(N, \{\vv^i_{s,t}\}_{i,s,t}, \{\tilde{z}_i\}_i, F')$ is a rational standard model for $\lambda$. In particular, $N$ is split with respect to $F'$. For any $g \in G$, if $\Ad_g(N)$ is $F'$-stable then $g^{-1}F'({g}) \in Z_G(N),$ and there exists a bijective correspondence
$$\{\textup{rational orbits in } (G\cdot N)^{F'}\} \leftrightarrow A(\lambda)$$
where to each rational orbit $G^{F'}\cdot \Ad_g(N) \subset (G\cdot N)^{F'}$ we assign the image of $g^{-1}F'({g})$ under  $Z_G(N) \twoheadrightarrow A(N) \simeq A(\lambda)$. (This is the result of Lang-Steinberg theorem, see \cite[Proposition 3.21]{dm91}.)

From now on, we say that $N$ is untwisted with respect to $F'$ if there exists a rational standard model $(N, \{\vv^i_{s,t}\}_{i,s,t}, \{\tilde{z}_i\}_i, F')$. Then it is clear that such an element exists and the set of untwisted elements with respect to $F'$ is a rational orbit. (An untwisted nilpotent element is split but not vice versa in general, see \cite[p.381]{sho97}.) Also for any $z\in A(\lambda)$ we say that $N$ is twisted by $z$ with respect to $F'$ if its rational orbit is parametrized by $z$ under the correspondence above. In particular, an untwisted nilpotent element is twisted by the identity. 

In fact, splitness of each $Q^i_{F'}$ determines the rational orbit of $N$. We have
\begin{prop} \label{prop:twisted} Let $N\in\fg$ be an $F'$-stable nilpotent element of Jordan type $(1^{m_1}2^{m_2} \cdots).$ Set $I =\{ i \in \bZ_{>0} \mid i \textup{ even, } m_i >0, Q^i_{F'} \textup{ non-split}\}$. Then $N$ is twisted by $\prod_{i \in I}z_i$ with respect to $F'$. \end{prop}
\begin{proof} It is easily deduced from \cite[\textsection 3]{sho83}. (See also \cite[3.2]{sho97}.)
\end{proof}
Combined with the argument in \ref{subsec:QF} we have the following corollary. Note that here we do not need the stronger assumption that $F'$ stabilizes all the elements in $\fB$.
\begin{cor} \label{cor:untwisted}For a basis $\fB$ adapted to $N$, if $F'$ stabilizes both $N$ and each element of $\fB \cap \ker N$ then $N$ is untwisted with respect to $F'$.
\end{cor}

Let us give more examples of twisted nilpotent elements. For simplicity, suppose $N\in \fg$ is a nilpotent element of Jordan type $(i^{m})$ for some even $i$ and assume that $(N, \{\vv_{s,t}\}_{s,t}, \{\tilde{z}\})$ is a standard model. We further assume that $F'(N)=N$ for some Frobenius $F'$, but we do not assume that each $\vv_{s,i}$ is $F'$-stable. Then we have the following.
\begin{lem} \label{lem:twistex} Let $z\in A(i^m)$ be the image of $\tilde{z}$ under $Z_G(N) \twoheadrightarrow A(N) \simeq A(i^m)$.
\begin{enumerate} 
\item Suppose that $m$ is odd and let $j = \frac{m+1}{2}$. If $F'(\vv_{s,i})=\vv_{s,i}$ for $i \neq j$ and $F'(\vv_{j,i})=-\vv_{j,i}$, then $N$ is twisted by $z$ with respect to $F'$.
\item If there exists $1\leq j < \frac{m+1}{2}$ such that $F'(\vv_{s,i})=\vv_{s,i}$ for $i \neq j$ or $m+1-j$, $F'(\vv_{j,i})=\vv_{m+1-j,i}$, and $F'(\vv_{m+1-j,i})=\vv_{j,i}$, then $N$ is twisted by $z$ with respect to $F'$.
\item In general, if $J \subset \{1, 2, \cdots, \floor{\frac{m+1}{2}}\}$ is the set of $j$ such that $F'$ acts on $\vv_{j,i}$ as in (1) or (2), then $N$ is twisted by $z^{|J|}$ with respect to $F'$.
\end{enumerate}
\end{lem}
\begin{proof} It follows from Proposition \ref{prop:twisted}.
\end{proof}

\begin{rmk} When $G=SO_{2n+1}$ or $G=SO_{2n}$, we instead consider $\tilde{G}$-orbits on $\fg$ and replace $A(N)$ with $\tilde{A}(N)$, $A(\lambda)$ with $\tilde{A}(\lambda)$, etc. Then the statement in this section is still valid if one switch the conditions of $i$ being even and $i$ being odd. 
\end{rmk}

\section{Partial Springer resolution} \label{sec:partial}
In this section, we develop a tool from the theory of partial Springer resolution studied in \cite{bm83}. It will be crucial for the proof of our main theorem. For a partial flag variety $\cP$ of $G$, we define
$$\tilde{\fg}^\cP \colonequals \{ (x, P) \in \fg \times \cP \mid x \in \Lie P\}$$
and also set $\tilde{\fg} \colonequals \tilde{\fg}^\cB$. Then the Grothendieck-Springer resolution $\pi: \tilde{\fg} \rightarrow \fg$ factors through $\tilde{\fg} \xrightarrow{\eta} \tilde{\fg}^\cP \xrightarrow{\xi} \fg$ where $\eta: \tilde{\fg} \rightarrow \tilde{\fg}^\cP$ is the natural projection. Let $W_\cP$ be the Weyl group of some/any parabolic subgroup $P \subset \cP$, considered as a corresponding parabolic subgroup of $W$. According to \cite[2.13]{bm83}, the complex $\eta_*\qlbar$ is equipped with a $W_\cP$-action such that its induced action on $\xi_*(\eta_*\qlbar) = \pi_* \qlbar$ coincides with the restriction of the usual Springer $W$-action.

For $x\in \fg$, let $\cB_x \xrightarrow{\eta_x} \cP_x \xrightarrow{\xi_x} \{x\}$ be the restriction of $\tilde{\fg} \xrightarrow{\eta} \tilde{\fg}^\cP \xrightarrow{\xi} \fg$ at $x$. Then $\cB_x$ is the Springer fiber of $x$ and
$\eta_{x*} \qlbar_{\cB_x}$  is equipped with a $W_\cP$-action which is compatible with the Springer $W$-action on $(\xi_{x} \circ\eta_{x})_* \qlbar_{\cB_x}$, i.e. cohomology groups of $\cB_x$. This implies that the Leray spectral sequence
$$H^i(\cP_x, \cH^j \eta_{x*} \qlbar_{\cB_x}) \Rightarrow H^{i+j} (\cB_x)$$
is a spectral sequence of $W_\cP$-modules where the $W_\cP$-action on $H^{i+j} (\cB_x)$ is given by the restriction of the usual Springer $W$-action. (See \cite[Theorem 5.1]{hs79} for its proof in characteristic 0 case.)

Let $N\in \fg$ be an $F$-stable nilpotent element.
For $z \in A(N)$, let us choose its representative $\tilde{z} \in Z_G(N)$ which is $F$-stable. (It is always possible by Lang-Steinberg theorem, see e.g. \cite[Corollary 3.13]{dm91}.) Also let $w\in W_\cP$. Our goal is to calculate $\tr(wzF^*, H^*(\cB_N))$. By the spectral sequence above, it is the same as
$$\sum_{i,j \in \bZ}(-1)^{i+j}\tr(w\tilde{z}^*F^*, H^i(\cP_N, \cH^j \eta_{N*} \qlbar_{\cB_N}))= \sum_{j\in \bZ}(-1)^{j}\tr(w\tilde{z}^*F^*, H^*(\cP_N, \cH^j \eta_{N*} \qlbar_{\cB_N})).$$

There is a canonical isomorphism $F^*: F^*\cH^j\eta_{N*} \qlbar_{\cB_N} \xrightarrow{\simeq} \cH^j\eta_{N*} \qlbar_{\cB_N} $ which comes from the natural $\bF_q$-structure of $\eta_{N*} \qlbar_{\cB_N}$. We consider the Weil sheaf $(\cH^j\eta_{N*} \qlbar_{\cB_N}, w\tilde{z}^*F^*)$, i.e. $\eta_{N*} \qlbar_{\cB_N}$ with the fixed isomorphism
$$w\tilde{z}F^*: F^*\cH^j\eta_{N*} \qlbar_{\cB_N} \xrightarrow{F^*} \cH^j\eta_{N*} \qlbar_{\cB_N} \xrightarrow{w\tilde{z}^*}\cH^j\eta_{N*} \qlbar_{\cB_N}.$$
Here the underlying geometric Frobenius morphism is given by $\tilde{z}F = \Ad_{\tilde{z}}\circ F: \cB_N \rightarrow \cB_N$. 
By the Grothendieck trace formula for Weil sheaves \cite[Corollary 1.5]{kw01}, we have
\begin{align*}
\tr(w\tilde{z}^*F^*, H^*(\cP_N, \cH^j \eta_{N*} \qlbar_{\cB_N}))&= \sum_{l \in (\cP_N)^{\tilde{z}F}} \tr(w\tilde{z}^*F^*, (\cH^j\eta_{N*} \qlbar_{\cB_N})_l) 
\\&= \sum_{l \in (\cP_N)^{\tilde{z}F}} \tr(w\tilde{z}^*F^*, H^j(\eta_N^{-1}(l)))
\end{align*}
Note that $\eta_N^{-1}(l)$ is the Springer fiber of $N$ with respect to the stabilizer of $l$ in $G$.

Therefore, it follows that
\begin{align*}
\tr(wzF^*, H^*(\cB_N)) &= \sum_{j\in \bZ}(-1)^{j}\sum_{l \in (\cP_N)^{\tilde{z}F}} \tr(w\tilde{z}^*F^*, H^j(\eta_N^{-1}(l)))
\\&= \sum_{l \in (\cP_N)^{\tilde{z}F}} \tr(w\tilde{z}^*F^*, H^*(\eta_N^{-1}(l))).
\end{align*}
In particular, if $\cP$ is the Grassmannian of isotropic lines, then $W_\cP = W'$, $\cP_N=\bP(\ker N)$, and $\eta^{-1}_N(l)$ for each $l \in \cP$ is the Springer fiber of $N|_{l^\perp/l}$. (Recall the definition of $W'$ in Section \ref{sec:setup}.) Thus for any $w \in W'$ we have
\begin{align*}
\tr(wzF^*, H^*(\cB_N))  = \sum_{l \in \bP(\ker N)^{\tilde{z}F}} \tr(w\tilde{z}^*F^*, H^*(\cB'_{N'}))
\end{align*}
where $\cB'_{N'}=\eta^{-1}_N(l)$ is the Springer fiber of $N' = N|_{l^\perp/l}$ with respect to $Sp(l^\perp/l)$ and the $W'$-action on $H(\cB'_{N'})$ is given by the Springer action of $W'$ considered as the Weyl group of $Sp(l^\perp/l)$.

Now assume that $N$ is split of Jordan type $\lambda$. Then by \cite[(4.1.2)]{sho83} (see also \cite[Lemma 6.3]{spr76}) we have
$$\tr(wzF^*, H^*(\cB_N)) = \tspq{\lambda, z}(w).$$
Similarly, if $\lambda(l)$ is the Jordan type of $N' = N|_{l^\perp/l}$ and if $N'$ is twisted by $z(l) \in A(\lambda(l))$ with respect to $\tilde{z}F$, then
$$ \tr(w\tilde{z}^*F^*, H^*(\cB'_{N'})) = \tspq{\lambda(l), z(l)}(w)$$
We summarize this as follows. This will be our primary tool to prove the main theorem.
\begin{prop}\label{prop:maintool} 
Suppose that $N \in \fg$ is an $F$-stable nilpotent element of Jordan type $\lambda$ which is split. Also suppose that $\tilde{z}\in Z_G(N)^F$ maps to $z\in A(\lambda)$ under $Z_G(N) \rightarrow A(N) \simeq A(\lambda)$. Then we have
\begin{align*}
\Res^W_{W'} \tspq{\lambda, z} = \sum_{l \in \bP(\ker N)^{\tilde{z}F}}  \tspq{\lambda(l), z(l)}
\end{align*}
where $\lambda(l)$ is the Jordan type of $N|_{l^\perp/l}$ and $z(l)$ is the element in $A(\lambda(l))$ such that $N|_{l^\perp/l}$ is twisted by $z(l)$ with respect to $\tilde{z}F$.
\end{prop}

\begin{rmk} Here we consider the case when $G=SO_{2n+1}$ or $SO_{2n}$.
\begin{enumerate}[label=\arabic*., leftmargin =*]
\item It is not necessarily true that $\cP_N = \bP(\ker N)$ especially when $N$ has a Jordan block of size $1$. However, the statement of Proposition \ref{prop:maintool} is still valid since if $l \in \bP(\ker N)$ is not isotropic then $\eta_N^{-1}(l)$ is empty.
\item The proposition above is still valid when $G=SO_{2n}$ and one of $\lambda$ or $\lambda(l)$ is very even. Indeed, if $\lambda$ is very even then $n$ is even, thus none of $\lambda(l)$ is very even. But in this case $\Res^W_{W'} \tspq{\lambda, z} = \Res^W_{W'} \tspq{\lambda+, z} = \Res^W_{W'} \tspq{\lambda-, z}$ for any $z\in A(\lambda)$, thus it does not matter whether the orbits of $N$ is parametrized by either $\lambda+$ or $\lambda-$. On the other hand, if one of $\lambda(l)$ is very even then $n$ is odd thus $\lambda$ is not very even. Therefore $\Res^W_{W'} \tspq{\lambda, z}$ is stable under the outer automorphism of $W'$ swapping $s_1$ and $s_2$, which means that the coefficients of $\tspq{\lambda(l)+, z(l)}$ and $\tspq{\lambda(l)-, z(l)}$ in the formula should be the same. Then the claim follows from the definition of $\tspq{\lambda(l),z(l)}$ for $\lambda(l)$ very even.
\end{enumerate}
\end{rmk}

\section{Rectangular case}\label{sec:rect}
We start with the case when the Jordan type of a nilpotent element is  $(i^m)$ for some $i, m \in \bZ_{>0}$, i.e. a rectangle. When $i$ is odd, let $(N, \{\vv_{s,t}\}_{s,t}, \emptyset, F)$ be a rational standard model for $(i^m)$. Similarly when $i$ is even, let $(N, \{\vv_{s,t}\}_{s,t}, \{\tilde{z}\}, F)$ be a rational standard model for $(i^m)$. In this section we relate $\Res^W_{W'}\tspq{(i^m), -}$ with $\tspq{((i-1)^2i^{m-2}), -}$ and $\tspq{((i-2)^1i^{m-1}), -}$ using Proposition \ref{prop:maintool}.

\subsection{$i$ odd} In this case $m$ is always even, and $A(N)$ is trivial. Also for any $l \in \bP(\ker N)$, $N|_{l^\perp/l}$ is always of Jordan type $((i-1)^2i^{m-2})$. For $v \in \ker N$ we write $v=x_1\vv_{1,i}+x_2\vv_{2,i}+\cdots + x_m\vv_{m,i}$ for some $x_1, \ldots, x_m \in \Fqbar$. Then we have a filtration
$$ \emptyset =X_m \subset X_{m-1}\subset \cdots \subset X_1 \subset X_0= \bP(\ker N)$$
where each $X_{j}$ is a subset of $\bP(\ker N)$ with the additional condition $x_1= x_2 = \cdots = x_j = 0$. 

Suppose $l \in \bP(\ker N)$ and $\spn{v} = l$ for some $v \in \ker N$. If $l \in X_{j-1} - X_j$ then $x_1= x_2 = \cdots = x_{j-1} = 0$ but $x_j \neq 0$. Thus without loss of generality we may set $x_j=1$. Then $l$ is $F$-stable if and only if $x_k \in \Fq$ for all $k \geq j+1$. We let 
\begin{gather*}
\ww_{s,t}^i \colonequals \vv_{s,t} - \vartheta_s x_{m+1-s} \vv_{m+1 - j, t} \quad \textup{ for } s\neq j \textup{ or } m+1-j, \textup{ and}
\\\ww^{i-1}_{1,t} \colonequals \vv_{j,t}+x_{j+1}\vv_{j+1,t}+\cdots+x_m \vv_{m,t}, \qquad \ww^{i-1}_{2,t}\colonequals \vv_{m+1-j,t+1}.
\end{gather*}
Here, $\vartheta_s=1$ if $s$ and $m+1-j$ are either both $\geq \frac{m}{2}$ or both $< \frac{m}{2}$, and $\vartheta_s=-1$ otherwise. Then direct calculation shows that
$$\fB\colonequals \{\ww^i_{s,t} \mid 1\leq s \leq m, s\neq j \textup{ or } m+1-j, 1\leq t \leq i\} \cup \{\ww^{i-1}_{s,t} \mid s=1 \textup{ or }2, 1\leq t \leq i-1\}$$
form a basis of $l^\perp/l$ adapted to $N|_{l^\perp/l}$, and each element in $\fB \cap \ker N|_{l^\perp/l}$ is $F$-stable. Thus by Corollary \ref{cor:untwisted}, $N|_{l^\perp/l}$ is untwisted with respect to $F$. 

Since $\# \bP(\ker N)^{F} =  ({q^m-1})/({q-1})$, by Proposition \ref{prop:maintool} we conclude that
$$\Res^W_{W'}\tspq{(i^m), id} = \frac{q^m-1}{q-1} \tspq{((i-1)^2i^{m-2}), id}.$$

\subsection{$i$ even, $m$ odd}  In this case $A(N) \simeq \bZ/2.$ We set $r \colonequals \frac{m+1}{2}$. Let $H \subset \bP(\ker N)$ be the hypersurface consisting of zeroes of $Q$ where $Q$ is defined as in \ref{subsec:Q}. Then the Jordan type of $N|_{l^\perp/l}$ is equal to $((i-1)^2 i^{m-2})$ (resp. $((i-2)^1 i^{m-1})$) if $l \in H$ (resp. $l \not\in H$).

For $v \in \ker N$ we write $v=x_1\vv_{1,i}+x_2\vv_{2,i}+\cdots + x_m\vv_{m,i}$ for some $x_1, \ldots, x_m \in \Fqbar$. Then $H$ is defined by the equation
$$2x_1x_m+2x_2x_{m-1}+ \cdots + 2x_{r-1}x_{r+1}+x_r^2=0.$$

\subsubsection{$l \in H$ case}
We define a filtration
$$ \emptyset =H_m \subset H_{m-1}\subset \cdots \subset H_{r+1} \subset H_{r} = H_{r-1} \subset H_{r-2} \subset \cdots \subset H_1 \subset H_0 = H$$
where $H_{j}$ is defined by the subset of $H$ with the additional condition $x_1= x_2 = \cdots = x_j = 0$. 

Suppose $l\in H$ and $\spn{v} =l$ for some $v \in \ker N$. If $l \in H_{j-1}- H_{j}$ then $j\neq r$ and $x_1= x_2 = \cdots = x_{j-1}= 0$ but $x_j \neq 0$. Thus without loss of generality we may set $x_j = 1$. First assume that $l$ is $F$-stable, i.e. $x_k \in \bF_q$ for all $k\geq j+1$. We let
\begin{gather*}
\ww_{s,t}^i \colonequals \vv_{s,t} - x_{m+1-s} \vv_{m+1 - j, t} \quad \textup{ if } s\neq j \textup{ or }m+1-j, \textup{ and}
\\\ww_{1,t}^{i-1} \colonequals  \vv_{j,t}+x_{j+1}\vv_{j+1,t}+\cdots+x_m \vv_{m,t}, \qquad \ww_{2,t}^{i-1} \colonequals \vv_{m+1-j,t+1}.
\end{gather*}
Then direct calculation shows that
$$\fB\colonequals \{\ww^i_{s,t} \mid 1\leq s \leq m, s\neq j \textup{ or } m+1-j, 1\leq t \leq i\} \cup \{\ww^{i-1}_{s,t} \mid s=1 \textup{ or }2, 1\leq t \leq i-1\}$$
is a basis of $l^\perp/l$ adapted to $N_{l^\perp/l}$, and each element in $\fB \cap \ker N|_{l^\perp/l}$ is $F$-stable. Thus by Corollary \ref{cor:untwisted}, $N|_{l^\perp/l}$ is untwisted with respect to $F$. Also note that $\#H^F = 1+q+\cdots  q^{m-2} = ({q^{m-1}-1})/({q-1})$.

This time we assume that $l \in H_{j-1}- H_{j}$ is $\tilde{z}F$-stable. This condition is equivalent to that $x_k \in \bF_q$ for all $k \neq r$ and $(x_{r})^q=-x_{r}$. (In particular, $x_r \in \Fqs$.) Then $\fB$ above is still adapted to $N|_{l^\perp/l}$ and each element in $\fB \cap \ker N|_{l^\perp/l}$ is $\tilde{z}F$-stable except that $\tilde{z}F(\ww^i_{r,i}) = -\vv_{r,i}+x_r \vv_{m+1-j,i} = -\ww^i_{r,i}.$ Thus by Lemma \ref{lem:twistex}, $N_{l^\perp/l}$ is twisted by $z_i \in A((i-1)^2 i^{m-2})$ with respect to $\tilde{z}F$. Also note that $\#H^{\tilde{z}F} = 1+q+\cdots  q^{m-2} = ({q^{m-1}-1})/({q-1})$.

\subsubsection{$l \not \in H$ case}
Define $C \colonequals \bP(\ker N) - H$. Similarly we have a filtration
$$ \emptyset =C_{r} \subset C_{r-1}\subset \cdots \subset C_1 \subset C_0 = C$$
where $C_{j}$ is defined by the subset of $C$ with the additional condition $x_1= x_2 = \cdots = x_j = 0$. We assume $l  \in C_{j-1} - C_j$ and $\spn{v} = l$ for some $v=x_1\vv_{1,i}+x_2\vv_{2,i}+\cdots + x_m\vv_{m,i} \in \ker N$. Then  $x_1= x_2 = \cdots = x_{j-1}= 0$ but $x_j \neq 0$. Thus without loss of generality we may put $x_j = 1$.

First we assume that $j \neq r$. We set 
\begin{align*}
\delta&\colonequals 2x_1 x_{m}+ 2x_2x_{m-1}+\cdots + 2x_{r-1}x_{r+1}+x_{r}^2
\\&=2x_{m+1-j}+2x_{j+1}x_{m-j}+\cdots + 2x_{r-1}x_{r+1}+x_{r}^2 \in \Fq
\end{align*}
which is nonzero by assumption. Also choose $\epsilon \in \Fqs^\times$ such that $\epsilon^2 = \delta$. 

Suppose that $l$ is $F$-stable, i.e. $x_k \in \Fq$ for all $k\geq j+1$. We define
\begin{gather*}
\ww_{s,t}^{i} \colonequals \vv_{s,t} - x_{m+1-s} \vv_{m+1 - j, t} \quad \textup{ if } s\not \in \{ j, r,m+1-j\}, 
\\\ww^{i-2}_{t}\colonequals \frac{1}{\ii\epsilon}(\vv_{j,t+1}+x_{j+1}\vv_{j+1,t+1}+\cdots+x_m \vv_{m,t+1}), \textup{ and}
\\ \ww_{\pm, t}^i \colonequals \pm \frac{ 1}{\sqrt{2} \epsilon} \left(\vv_{j,t}+x_{j+1}\vv_{j+1,t}+\cdots+x_m \vv_{m,t}\pm \epsilon \vv_{r,t} - (\delta \pm \epsilon x_r) \vv_{m+1-j,t}\right).
\end{gather*}
Then direct calculation shows that 
$$\fB = \{\ww^i_{s,t} \mid s \in \{1, \ldots, m, +, -\}-\{j, r, m+1-j\}, 1\leq t \leq i\} \cup \{\ww^{i-2}_t \mid 1\leq t \leq i-2\}$$
is a basis of $l^\perp/l$ adapted to $N|_{l^\perp/l}$.

Now there are two cases to consider. If $\epsilon \in \Fq^\times$, then each element in $\fB \cap \ker N|_{l^\perp/l}$ is $F$-stable, thus $N|_{l^\perp/l}$ is untwisted with respect to $F$. Otherwise, i.e. if $\epsilon \in \Fqs-\Fq$, then $F(\epsilon) = -\epsilon$ since $(\epsilon^{q-1})^2 = \delta^{q-1}=1$ but $\epsilon^{q-1} \neq 1$. Therefore each element in $\fB \cap \ker N|_{l^\perp/l}$ is $F$-stable except that $F$ swaps $\ww^i_{+,t}$ and $\ww^i_{-,t}$ and $F(\ww_t^{i-2}) = -\ww_t^{i-2}$. By Lemma \ref{lem:twistex}, we see that $N|_{l^\perp/l}$ is twisted by $z_i z_{i-2} \in A((i-2)^1i^{m-1})$ with respect to $F$.

When $l$ is $\tilde{z}F$-stable, then $x_k \in \bF_q$ for all $k \neq r$ and $(x_{r})^q=-x_{r}$. We choose a basis $\fB$ the same as before except
\begin{gather*}
\ww^{i}_{\pm, i}\colonequals \frac{ 1}{\sqrt{2}\ii \epsilon} \left(\vv_{j,t}+x_{j+1}\vv_{j+1,t}+\cdots+x_m \vv_{m,t}\pm \epsilon \vv_{r,t} - (\delta \pm \epsilon x_r) \vv_{m+1-j,t}\right).
\end{gather*}
Direct calculation shows that $\fB$ is a basis of $l^\perp/l$ adapted to $N|_{l^\perp/l}$. In this case, if $\epsilon \in \Fq^\times$ then $\tilde{z}F$ stabilizes each element in $\fB \cap \ker N|_{l^\perp/l}$ except that $\tilde{z}F$ swaps $\ww^{i}_{+, i}$ and $\ww^{i}_{-, i}$. Thus $N|_{l^\perp/l}$ is twisted by $z_i \in A((i-2)^1i^{m-1})$ with respect to $\tilde{z}F$. On the other hand, if $\epsilon \in \Fqs-\Fq$ then $F(\epsilon)=-\epsilon$ by the same reason above. In this situation $\tilde{z}F$ stabilizes each element in $\fb \cap \ker N|_{l^\perp/l}$ except that $\tilde{z}F(\ww^{i-2}_t) = -\ww^{i-2}_t$. This implies that $N|_{l^\perp/l}$ is twisted by $z_{i-2} \in A((i-2)^1i^{m-1})$ with respect to $\tilde{z}F$.

It remains to check the case when $l \in C_{r-1} - C_r$. As before we set $x_r =1$. Suppose that $l$ is $F$-stable, i.e. $x_k \in \Fq$ for all $k\geq r+1$.  We define
\begin{gather*}
\ww_{s,t}^i \colonequals \vv_{s,t} - x_{m+1-s} \vv_{r, t}- \frac{x_{m+1-s}^2}{2} \vv_{m+1-s,t} \quad \textup{ if } s\neq r, \textup{ and}
\\\ww^{i-2}_{t}\colonequals \ii(\vv_{r,t+1}+x_{r+1}\vv_{r+1,t+1}+\cdots+x_m \vv_{m,t+1}).
\end{gather*}
Then the set
$$\fB = \{\ww^i_{s,t} \mid 1\leq s \leq m, s\neq r, 1\leq t \leq i\} \cup \{\ww^{i-2}_t \mid 1\leq t \leq i-2\}$$
is a basis of $l^\perp/l$ adapted to $N|_{l^\perp/l}$. Also, each element in $\fB \cap \ker N|_{l^\perp/l}$ is $F$-stable. Thus $N|_{l^\perp/l}$ is untwisted with respect to $F$.

This time we assume that $l$ is $\tilde{z}F$-stable, i.e. $x_k^q = -x_k$ for all $k \geq r+1$. Then the same $\fB$ is a basis of $l^\perp/l$ adapted to $N|_{l^\perp/l}$, and each element in $\fB \cap \ker N|_{l^\perp/l}$ is $\tilde{z}F$-stable except that $\tilde{z}F(\ww^{i-2}_{i-2}) = -\ww^{i-2}_{i-2}$. Thus in this case $N|_{l^\perp/l}$ is twisted by $z_{i-2} \in A((i-2)^1i^{m-1})$ with respect to $\tilde{z}F$.

In order to apply Proposition \ref{prop:maintool}, it remains to calculate the cardinals of the following sets.
\begin{align*}
X&\colonequals\{l \in C^F \mid \textup{ $N|_{l^\perp/l}$ is untwisted with respect to $F$}\},
\\Y&\colonequals\{l \in C^F \mid \textup{ $N|_{l^\perp/l}$ is twisted by $z_i z_{i-2} \in A((i-2)^1i^{m-1})$ with respect to $F$}\},
\\X'&\colonequals\{l \in C^{\tilde{z}F} \mid \textup{ $N|_{l^\perp/l}$ is twisted by $z_{i-2} \in A((i-2)^1i^{m-1})$ with respect to $\tilde{z}F$}\},
\\Y'&\colonequals\{l \in C^{\tilde{z}F} \mid \textup{ $N|_{l^\perp/l}$ is twisted by $z_{i} \in A((i-2)^1i^{m-1})$ with respect to $\tilde{z}F$}\}.
\end{align*}
We first calculate $\#X$ and $\#Y$. Note that $\#X+\#Y = \#C^F = q^{m-1}$. It is clear that $\#(X \cap (C_{j-1}-C_j)) = \#(Y\cap (C_{j-1}-C_j))$ when $j>r$ since the number of quadratic residues in $\Fq^\times$ equals $|\Fq^\times|/2=(q-1)/{2}$. On the other hand, $(C_{r-1}-C_r)^F \subset X$ as we observed above. Since $\#(C_{r-1}-C_r)^F = q^{r-1}$, we have $\#X = ({q^{m-1}+q^{r-1}})/{2}$ and $\#Y =({q^{m-1}-q^{r-1}})/{2}$. Similar argument applies to $X'$ and $Y'$, and we have $\#X' =  ({q^{m-1}+q^{r-1}})/{2}$ and $Y' = ({q^{m-1}-q^{r-1}})/{2}$.

Therefore by Proposition \ref{prop:maintool} we have
\begin{align*}
\tspq{(i^m), id} =\ &\frac{q^{m-1}-1}{q-1}\tspq{((i-1)^2 i^{m-2}), id}
\\&+\frac{q^{m-1}+q^{r-1}}{2}\tspq{((i-2)^1i^{m-1}),id}
\\&+\frac{q^{m-1}-q^{r-1}}{2}\tspq{((i-2)^1i^{m-1}),z_{i-2}z_i},
\\\tspq{(i^m), z_i} =\ &\frac{q^{m-1}-1}{q-1}\tspq{((i-1)^2 i^{m-2}), z_i}
\\&+\frac{q^{m-1}+q^{r-1}}{2}\tspq{((i-2)^1i^{m-1}),z_{i-2}}
\\&+\frac{q^{m-1}-q^{r-1}}{2}\tspq{((i-2)^1i^{m-1}),z_i}.
\end{align*}
If we combine these two formulas we get
\begin{align*}
\tspq{(i^m), z_i^a} =\ &\frac{q^{m-1}-1}{q-1}\tspq{((i-1)^2 i^{m-2}), z_i^a}
\\&+\frac{q^{m-1}+q^{\frac{m-1}{2}}}{2}\tspq{((i-2)^1i^{m-1}),z_{i-2}^a}
\\&+\frac{q^{m-1}-q^{\frac{m-1}{2}}}{2}\tspq{((i-2)^1i^{m-1}),z_iz_{i-2}^{a+1}}
\end{align*}
for any $a \in \bZ$.

\subsection{$i$ even, $m$ even}  In this case $A(N) \simeq \bZ/2.$ We argue similarly to $i$ even, $m$ odd case. Let $r\colonequals \frac{m}{2}$. Let $H \subset \bP(\ker N)$ be the hypersurface consisting of zeroes of $Q$ where $Q$ is defined as in \ref{subsec:Q}. Then the Jordan type of $N|_{l^\perp/l}$ is equal to $((i-1)^2 i^{m-2})$ (resp. $((i-2)^1 i^{m-1})$) if $l \in H$ (resp. $l \not\in H$).

For $v \in \ker N$ we write $v=x_1\vv_{1,i}+x_2\vv_{2,i}+\cdots + x_m\vv_{m,i}$ for some $x_1, \ldots, x_m \in \Fqbar$. Then $H$ is defined by the equation
$$2x_1x_m+2x_2x_{m-1}+ \cdots + 2x_{r}x_{r+1}=0.$$

\subsubsection{$l \in H$ case}
We define a filtration
$$ \emptyset =H_m \subset H_{m-1}\subset \cdots \subset H_{r+1} \subset H_{r} \subset H_{r-1}  \subset \cdots \subset H_1 \subset H_0 = H$$
where $H_{j}$ is defined by the subset of $H$ with the additional condition $x_1= x_2 = \cdots = x_j = 0$. 

Suppose $l\in H$ and $\spn{v} =l$ for some $v \in \ker N$. If $l \in H_{j-1}- H_{j}$ then $x_1= x_2 = \cdots = x_{j-1}= 0$ but $x_j \neq 0$. Thus without loss of generality we may set $x_j = 1$. First assume that $l$ is $F$-stable, i.e. $x_k \in \bF_q$ for all $k\geq j+1$. We let
\begin{gather*}
\ww_{s,t}^i \colonequals \vv_{s,t} - x_{m+1-s} \vv_{m+1 - j, t} \quad \textup{ if } s\neq j \textup{ or }m+1-j, \textup{ and}
\\\ww_{1,t}^{i-1} \colonequals  \vv_{j,t}+x_{j+1}\vv_{j+1,t}+\cdots+x_m \vv_{m,t}, \qquad \ww_{2,t}^{i-1} \colonequals \vv_{m+1-j,t+1}
\end{gather*}
Then direct calculation shows that
$$\fB\colonequals \{\ww^i_{s,t} \mid 1\leq s \leq m, s\neq j \textup{ or } m+1-j, 1\leq t \leq i\} \cup \{\ww^{i-1}_{s,t} \mid s=1 \textup{ or }2, 1\leq t \leq i-1\}$$
is a basis of $l^\perp/l$ adapted to $N_{l^\perp/l}$, and each element in $\fB \cap \ker N|_{l^\perp/l}$ is $F$-stable. Thus by Corollary \ref{cor:untwisted}, $N|_{l^\perp/l}$ is untwisted with respect to $F$. Also note that $\#H^F = 1+q+\cdots + q^{r-2}+q^{r-1} + q^{r-1}+q^{r}+\cdots + q^{m-2} = ({q^{m-1}-1})/({q-1})+q^{r-1}$.

This time we assume that $l$ is $\tilde{z}F$-stable. This assumption forces that $j\not\in \{r, r+1\}$. Then we have $x_k \in \bF_q$ for all $k \not\in \{r-1, r\}$, $x_r^q=x_{r+1}$, and $x_{r+1}^q=x_r$. (Thus in particular $x_r, x_{r+1} \in \Fqs$.) We set $\fB$ to be the same as above. Then $\fB$ is a basis of $l^\perp/l$ adapted to $N|_{l^\perp/l}$, and $\tilde{z}F$ stabilizes each element in $\fB \cap \ker N|_{l^\perp/l}$ except that $\tilde{z}F$ swaps $\ww^i_{r,i}$ and $\ww^i_{r+1,i}$. Thus by Lemma \ref{lem:twistex}, $N_{l^\perp/l}$ is twisted by $z_i \in A((i-1)^2 i^{m-2})$ with respect to $\tilde{z}F$. Also note that $\#H^{\tilde{z}F} = 1+q+\cdots + q^{r-2}+q^{r}+\cdots + q^{m-2} = ({q^{m-1}-1})/({q-1})-q^{r-1}$.

%
%
%
%
%

\subsubsection{$l \not \in H$ case}
Define $C \colonequals \bP(\ker N) - H$. Similarly we have a filtration
$$ \emptyset =C_{r} \subset C_{r-1}\subset \cdots \subset C_1 \subset C_0 = C$$
where $C_{j}$ is defined by the subset of $C$ with the additional condition $x_1= x_2 = \cdots = x_j = 0$. We assume $l  \in C_{j-1} - C_j$ and $\spn{v} = l$ for some $v=x_1\vv_{1,i}+x_2\vv_{2,i}+\cdots + x_m\vv_{m,i} \in \ker N$. Then  $x_1= x_2 = \cdots = x_{j-1}= 0$ but $x_j \neq 0$. Thus without loss of generality we may put $x_j = 1$.

First we assume that $j \neq r$. We set 
\begin{align*}
\delta&\colonequals 2x_1 x_{m}+ 2x_2x_{m-1}+\cdots + 2x_{r}x_{r+1}
\\&=2x_{m+1-j}+2x_{j+1}x_{m-j}+\cdots + 2x_{r}x_{r+1} \in \Fq
\end{align*}
which is nonzero by assumption. Also choose $\epsilon \in \Fqs^\times$ such that $\epsilon^2 = \delta$. 

Suppose that $l$ is $F$-stable, i.e. $x_k \in \Fq$ for all $k \geq j+1$. We define
\begin{gather*}
\ww_{s,t}^{i} \colonequals \vv_{s,t} - x_{m+1-s} \vv_{m+1 - j, t} \quad \textup{ if } s\not \in \{ j, m+1-j\}, 
\\\ww_{m+1-j,t}^i \colonequals \frac{1}{\ii \epsilon} (\vv_{j,t}+x_{j+1}\vv_{j+1,t}+\cdots+x_m \vv_{m,t} - \delta \vv_{m+1-j,t}), \quad \textup{ and}
\\ \ww_{t}^{i-2}\colonequals \frac{1}{\ii \epsilon}(\vv_{j,t+1}+x_{j+1}\vv_{j+1,t+1}+\cdots+x_m \vv_{m,t+1})
\end{gather*}
Then direct calculation shows that 
$$\fB = \{\ww^i_{s,t} \mid 1\leq s \leq m, s\neq j, 1\leq t \leq i\} \cup \{\ww^{i-2}_t \mid 1\leq t \leq i-2\}$$
is a basis of $l^\perp/l$ adapted to $N|_{l^\perp/l}$.

There are two cases to consider. If $\epsilon \in \Fq^\times$, then each element in $\fB \cap \ker N|_{l^\perp/l}$ is $F$-stable, thus $N|_{l^\perp/l}$ is untwisted with respect to $F$. Otherwise $F(\epsilon) = -\epsilon$, thus each element in $\fB \cap \ker N|_{l^\perp/l}$ is $F$-stable except that $F(\ww^i_{m+1-j,i})=-\ww^i_{m+1-j,i}$ and $F(\ww^{i-2}_{i-2}) = -\ww^{i-2}_{i-2}$. By Lemma \ref{lem:twistex}, we see that $N|_{l^\perp/l}$ is twisted by $z_i z_{i-2} \in A((i-2)^1i^{m-1})$ with respect to $F$.

When $l$ is $\tilde{z}F$-stable, then $x_k \in \bF_q$ for all $k \neq r$, $x_r^q=x_{r+1}$, and $x_{r+1}^q=x_r$. (In particular, $x_r, x_{r+1} \in \Fqs$.) Then the same $\fB$ as above is a basis of $l^\perp/l$ adapted to $N|_{l^\perp/l}$. In this case, if $\epsilon \in \Fq^\times$ then $\tilde{z}F$ stabilizes each element in $\fB \cap \ker N|_{l^\perp/l}$ except that $\tilde{z}F$ swaps $\ww^{i}_{r, i}$ and $\ww^{i}_{r+1, i}$. Thus $N|_{l^\perp/l}$ is twisted by $z_i \in A((i-2)^1i^{m-1})$ with respect to $\tilde{z}F$. On the other hand, if $\epsilon \in \Fqs-\Fq$ then $\tilde{z}F$ stabilizes each element in $\fB \cap \ker N|_{l^\perp/l}$ except that $\tilde{z}F$ swaps $\ww^{i}_{r, i}$ and $\ww^{i}_{r+1, i}$, $F(\ww^i_{m+1-j,i})=-\ww^i_{m+1-j,i}$, and $F(\ww^{i-2}_{i-2}) = -\ww^{i-2}_{i-2}$. By Lemma \ref{lem:twistex}, this means that $N|_{l^\perp/l}$ is twisted by $z_{i-2} (=z_i^2z_{i-2}) \in A((i-2)^1i^{m-1})$ with respect to $\tilde{z}F$.

It remains to check the case when $l \in C_{r-1} - C_r$. As before we set $x_r =1$. Suppose that $l$ is $F$-stable, i.e. $x_k \in \Fq$ for all $k \geq j+1$.  We define
\begin{gather*}
\ww_{s,t}^i \colonequals \vv_{s,t} - x_{m+1-s} \vv_{r+1, t} \quad \textup{ if } s\not \in \{ r, r+1\},
\\\ww_{r+1,t}^i \colonequals \frac{1}{\ii \epsilon} (\vv_{r,t}-x_{r+1}\vv_{r+1,t}+x_{r+2}\vv_{r+2,t}+\cdots+x_m \vv_{m,t} ), \textup{ and}
\\\ww^{i-2}_{t}\colonequals \frac{1}{\ii \epsilon}(\vv_{r,t+1}+x_{r+1}\vv_{r+1,t+1}+\cdots+x_m \vv_{m,t+1}).
\end{gather*}
Here $\epsilon$ is an element in $\Fqs$ satisfying $\epsilon^2 = 2x_{r+1}$.

Then the set
$$\fB = \{\ww^i_{s,t} \mid 1\leq s \leq m, s\neq r, 1\leq t \leq i\} \cup \{\ww^{i-2}_t \mid 1\leq t \leq i-2\}$$
is a basis of $l^\perp/l$ adapted to $N|_{l^\perp/l}$. If $\epsilon \in \Fq^\times$, then each element in $\fB \cap \ker N|_{l^\perp/l}$ is $F$-stable, thus $N|_{l^\perp/l}$ is untwisted with respect to $F$. Otherwise, the same is true except that $F(\ww_{r+1,i}^i)=- \ww_{r+1,i}^i$ and $F(\ww_{i-2}^{i-2})=-\ww_{i-2}^{i-2}$, thus $N|_{l^\perp/l}$ is twisted by $z_iz_{i-2} \in A((i-2)^1i^{m-1})$ with respect to $F$.

Now assume that $l$ is $\tilde{z}F$-stable. Then we have $x_{r+1}^{q+1}=1$ and $x_k^qx_{r+1} = x_k$ for $k \not \in \{ r, r+1\}$. First suppose that $x_{r+1} \in \Fq$, i.e. $x_{r+1} = \pm 1$. Then we set
\begin{gather*}
\ww^{i}_{s,t} \colonequals \vv_{s,t} -\frac{x_{m+1-s}}{2x_{r+1}}(\vv_{r,t} + x_{r+1}\vv_{r+1,t}) - \frac{x_{m+1-s}^2}{4x_{r+1}} \vv_{m+1-s, t} \quad \textup{ if } s \not \in \{r, r+1\},
\\\ww_{r+1,t}^{i} \colonequals  \frac{1}{\ii\epsilon}(\vv_{r,t}-x_{r+1}\vv_{r+1,t}), \textup{ and}
\\\ww_{t}^{i-2} \colonequals  \frac{1}{\ii\epsilon}(\vv_{r,t+1}+x_{r+1}\vv_{r+1,t+1}+\cdots+x_m \vv_{m,t+1})
\end{gather*}
where $\epsilon$ is an element in $\Fqs$ such that $\epsilon^2 = 2x_{r+1}$. Then $\fB$ is a basis of $l^\perp/l$ adapted to $N|_{l^\perp/l}$ and each element in $\fB \cap \ker N|_{l^\perp/l}$ is $\tilde{z}F$-stable possibly except $\ww^i_{r+1, i}$ and $\ww^{i-2}_{i-2}$. Indeed, we have
$$
\tilde{z}F(\ww_{r+1,i}^i) = \frac{1}{\ii \epsilon^q}(\vv_{r+1,t}-x_{r+1}^q\vv_{r,t}) = \frac{-x_{r+1}^q}{\ii \epsilon^q}(\vv_{r,t} - x_{r+1} \vv_{r+1,t}) = \frac{-1}{\epsilon^{q-1}x_{r+1}}\ww^i_{r+1,i}
$$
and similarly $\tilde{z}F(\ww^{i-2}_{i-2}) = \frac{1}{\epsilon^{q-1}x_{r+1}}\ww^{i-2}_{i-2}$. But since $\epsilon^{2(q+1)} = 2^{q+1}x_{r+1}^{q+1} = 4$, we have $\epsilon^{q+1} = \pm 2$ and $\epsilon^{q-1}x_{r+1} = \frac{\epsilon^{q+1}}{2} = \pm 1.$ In other words, if $\epsilon^{q+1} = 2$ then $N|_{l^\perp/l}$ is twisted by $z_{i} \in  A((i-2)^1i^{m-1})$ with respect to $\tilde{z}F$, and otherwise it is twisted by $z_{i-2} \in  A((i-2)^1i^{m-1})$ with respect to $\tilde{z}F$.

In order to apply Proposition \ref{prop:maintool}, we need to calculate the cardinals of the following sets.
\begin{align*}
X&\colonequals\{l \in C^F \mid \textup{ $N|_{l^\perp/l}$ is untwisted with respect to $F$}\},
\\Y&\colonequals\{l \in C^F \mid \textup{ $N|_{l^\perp/l}$ is twisted by $z_i z_{i-2} \in A((i-2)^1i^{m-1})$ with respect to $F$}\},
\\X'&\colonequals\{l \in C^{\tilde{z}F} \mid \textup{ $N|_{l^\perp/l}$ is twisted by $z_{i-2} \in A((i-2)^1i^{m-1})$ with respect to $\tilde{z}F$}\},
\\Y'&\colonequals\{l \in C^{\tilde{z}F} \mid \textup{ $N|_{l^\perp/l}$ is twisted by $z_{i} \in A((i-2)^1i^{m-1})$ with respect to $\tilde{z}F$}\}.
\end{align*}
We first calculate $\#X$ and $\#Y$. Note that $\#X+\#Y = \#C^F = q^{m-1}-q^{r-1}$. Also it is clear that $\#(X \cap (C_{j-1}-C_j)) = \#(Y\cap (C_{j-1}-C_j))$ for $j\geq r$ since the number of quadratic residues in $\Fq^\times$ equals $|\Fq^\times|/2=({q-1})/{2}$. Thus $\#X = \#Y = ({q^{m-1}-q^{r-1}})/{2}$. On the other hand, it is also clear that $\#X'+\#Y' = \#C^{\tilde{z}F} = q^{m-1}+q^{r-1}$ and $\#(X' \cap (C_{j-1}-C_j)) = \#(Y'\cap (C_{j-1}-C_j))$ for $j>r$. Now we claim that $\#(X' \cap (C_{r-1}-C_r)) = \#(Y'\cap (C_{r-1}-C_r))$. Indeed, if we choose $a\in \Fqs^\times$ such that $a^q=-a$ (which always exists), then
$$\spn{(0, \ldots, 0, 1, x_{r+1}, x_{r+2}, \ldots, x_m)} \mapsto \spn{(0, \ldots, 0, 1, -x_{r+1}, ax_{r+2}, \ldots, ax_m)}$$
gives a bijection between $X' \cap (C_{r-1}-C_r)$ and $Y'\cap (C_{r-1}-C_r)$. Therefore, we have $\#X' = \#Y' = ({q^{m-1}+q^{r-1}})/{2}$.

Therefore by Proposition \ref{prop:maintool} we have
\begin{align*}
\tspq{(i^m), id} =\ & \left(\frac{q^{m-1}-1}{q-1}+q^{r-1}\right)\tspq{((i-1)^2 i^{m-2}), id}
\\&+\frac{q^{m-1}-q^{r-1}}{2}\tspq{((i-2)^1i^{m-1}),id}
\\&+\frac{q^{m-1}-q^{r-1}}{2}\tspq{((i-2)^1i^{m-1}),z_{i-2}z_i},
\\\tspq{(i^m), z_i} =\ & \left(\frac{q^{m-1}-1}{q-1}-q^{r-1}\right)\tspq{((i-1)^2 i^{m-2}), z_i}
\\&+\frac{(q^{m-1}+q^{r-1})}{2} \tspq{((i-2)^1i^{m-1}),z_{i-2}}
\\&+\frac{(q^{m-1}+q^{r-1})}{2} \tspq{((i-2)^1i^{m-1}),z_i}.
\end{align*}
If we combine these two formulas we get
\begin{align*}
\tspq{(i^m), z_i^a} =\ & \left(\frac{q^{m-1}-1}{q-1}+(-1)^aq^{\frac{m}{2}-1}\right)\tspq{((i-1)^2 i^{m-2}), z_i^a}
\\&+\frac{q^{m-1}-(-1)^aq^{\frac{m}{2}-1}}{2}\tspq{((i-2)^1i^{m-1}),z_{i-2}^a}
\\&+\frac{q^{m-1}-(-1)^aq^{\frac{m}{2}-1}}{2}\tspq{((i-2)^1i^{m-1}),z_{i}z_{i-2}^{a+1}}
\end{align*}
for any $a\in \bZ$.
\begin{rmk} When $G=SO_{2n+1}$ or $SO_{2n}$, $z_i$ is not necessarily in $A(i^m)$ but in $\tilde{A}(i^m)$, thus $\tspq{i^m, z_i}$ may not be well-defined. However, as we will see, this does not cause any problem in the next section.
\end{rmk}

\section{Generalization}\label{sec:gen}
Here, we discuss how to apply the results in the previous section to a nilpotent element of a general Jordan type. We start with the following lemma. Let $F'\colonequals \Ad_h\circ F$ for some $h\in G$ and let $N \in \fg$ be a $F'$-stable nilpotent element of Jordan type $(1^{m_1}2^{m_2} \cdots)$. Consider the decomposition $V=V_1\oplus V_2 \oplus \cdots$ given in \ref{subsec:decomp} such that each $V_i$ is $F'$-stable. Fix some $i \in \bZ_{>0}$ such that $m_i \neq 0$.

\begin{lem} Let $0\neq x \in (\ker_i N)^{F'}$ and $v \in (\ker_{>i} N)^{F'}$. Then there exists $g \in Z_G(N)^{F'}$ such that $g(x)= x+v$.
\end{lem}

\begin{proof}
Write $v = \sum_{j} v_j$ where $v_j \in V_j$. First we claim that we may restrict our situation to $V'=V_i \oplus \bigoplus_{j}V_j$ where the direct sum $\bigoplus_{j}V_j$ is over all $j$ such that $v_j \neq 0$. To this end, let $V''=\bigoplus_{j}V_j$ which is over all $j$ such that $v_j = 0$ so that $V'\oplus V''$ is an orthogonal decomposition of $V$. Clearly $F'$ and $N$ preserves each $V'$ and $V''$, thus if we let $G'\colonequals Sp(V')$ and $G''\colonequals Sp(V'')$ then $N \in \Lie (G'\times G'') \subset \Lie G$ and $F'$ restricts to $G', \Lie G',$ etc. Now if there exists $g \in Z_{G'}(N|_{V'})^{F'}$ such that $g(x) = x+v$, then $(g, id) \in G'^{F'}\times G''^{F'} \subset G^{F'}$ sends $x$ to $x+v$, from which the claim follows.

Now we first consider the case when $V= V_j \oplus V_i$ for some $j > i$, i.e. the Jordan type of $N$ is $(i^{m_i}j^{m_j})$. The $v=0$ case is trivial, thus suppose otherwise and choose $x' \in V_i^{F'}, v' \in V_j^{F'}$ such that $N^{i-1}x' = x$ and $N^{j-1}v' = v$. (Such elements always exist.) Furthermore we may assume that $\br{N^{a}x', N^{b}x'}=0$ unless $a+b=i-1$ and $\br{N^{a}v', N^{b}v'}=0$ unless $a+b=j-1$. (This can easily be shown by using a standard model.) We consider the following two cases.
\begin{enumerate}
\item Restriction of $\br{\ ,\ }$ on $\spn{x', Nx', \ldots, N^{i-1}x'}$ is identically zero. We choose $y' \in V_i^{F'}$ such that 
\begin{gather*}
\br{N^a x', N^{b}y'}=0 \textup{ unless } a+b=i-1,
\\(-1)^k\br{N^k x', N^{i-1-k}y'}=1 \textup{ for } 0\leq k \leq i-1, \textup{ and}
\\\br{N^ay',N^by'}=0 \textup{ for any } a,b \in\bN.
\end{gather*}
This is always possible, which can be shown by using a standard model. If we let $V_i' = \spn{x', Nx', \ldots, N^{i-1}x', y', Ny', \ldots, N^{i-1}y'}$, then $V_i'^\perp\cap V_i'=0$ and we have an orthogonal decomposition
$$V\colonequals V_j \oplus V_i' \oplus (V_i'^\perp\cap V_i)$$
which is both $N$-stable and $F'$-stable. Now by similar reason to above we may assume that $V_i'^\perp\cap V_i=0$ and $V_i = V_i'$.
\item Otherwise, set $V_i' \colonequals \spn{x', Nx', \ldots, N^{i-1}x'}$. Then $V_i'^\perp\cap V_i' = 0$ and we have an orthogonal decomposition
$$V\colonequals V_j \oplus V_i' \oplus (V_i'^\perp\cap V_i)$$
which is both $N$-stable and $F'$-stable. Now by similar reason to above we may assume that $V_i'^\perp \cap V_i=0$ and $V_i=V_i'$.
\end{enumerate}
We apply the same argument to $v' \in V_j$. More precisely, if $\br{\ ,\ }$ on $\spn{v', Nv', \ldots, N^{j-1}v'}$ is identically zero then we choose $w' \in V_j^{F'}$ such that 
\begin{gather*}
\br{N^a v', N^{b}w'}=0 \textup{ unless } a+b=j-1,
\\(-1)^k\br{N^k v', N^{j-1-k}w'}=1, \textup{ for } 0 \leq k \leq j-1 \textup{ and}
\\\br{N^aw',N^bw'}=0 \textup{ for any } a,b \in\bN.
\end{gather*}
Then we may assume that $V_j=\spn{v', Nv', \ldots, N^{j-1}v', w', Nw', \ldots, N^{j-1}w'}$. Otherwise, we may assume that $V_j = \spn{v', Nv', \ldots, N^{j-1}v'}$. 

Now it suffices to consider the following four cases.
\begin{enumerate}[label=$\bullet$,  leftmargin =*] 
\item $m_i = m_j=2$. We define $g: V \rightarrow V$ to be
$$g(y') = y',\quad g(v') = v',\quad g(x') = x'+N^{j-i}v',\quad g(w') = w'- (-1)^{j-i} y'$$
which is extended by linearity and the condition $\Ad_g(N) = N$.
\item $m_i=2, m_j=1$. In this case $j$ is always even. We define $g: V \rightarrow V$ to be
\begin{gather*}
g(y') = y',\quad g(x') = x'+N^{j-i}v'+\frac{(-1)^{i-1}\br{v', N^{j-1}v'}}{2}N^{j-i}y',
\\ g(v') = v'+(-1)^{i-1}\br{v', N^{j-1}v'}y'
\end{gather*}
which is extended by linearity and the condition $\Ad_g(N) = N$.
\item $m_i=1, m_j=2$. In this case $i$ is always even. We define $g: V \rightarrow V$ to be
$$g(v') = v', \quad g(x') = x'+N^{j-i}v', \quad g(w') = w' + \frac{(-1)^{j-1}}{\br{x', N^{j-1}x'}}x'$$
which is extended by linearity and the condition $\Ad_g(N) = N$.
\item $m_i=m_j=1$. In this case both $i$ and $j$ are even. Define $a_k, b_k, c_k, \ldots \in \Fq$ for $k \in \bN$ inductively such that they satisfy
\begin{gather*}
a_0=1, \quad a_1 = \frac{\br{v', N^{j-1}v'}}{2\br{x', N^{i-1}x'}}, \quad \sum_{s=0}^k a_s a_{k-s}=0 \textup{ for } k \geq 2,
\\ c_0 = 1,\quad  \br{x', N^{i-1}x}\sum_{s=0}^{k-1}b_sb_{k-1-s}+\br{v', N^{i-1}v'}\sum_{s=0}^kc_sc_{k-s}=0 \textup{ for } k\geq 1, \textup{ and}
\\  \br{x', N^{i-1}x}\sum_{s=0}^k b_sa_{k-s}+\br{v', N^{i-1}v'}c_k =0 \textup{ for } k \geq 0.
\end{gather*}
Then we define $g: V\rightarrow V$ to be
\begin{gather*}
g(x') = x'+N^{j-i}v' + \sum_{k\geq 1}a_k N^{k(j-i)}x'
\\g(v') = v'+\sum_{k\geq 0} b_k N^{k(j-i)}x' + \sum_{k\geq 1}c_k N^{k(j-i)}v'
\end{gather*}
which is extended by linearity and the condition $\Ad_g(N) = N$.
\end{enumerate}
In each case, direct calculation shows that $g$ is $F'$-stable, $g$ preserves $\br{\ ,\ }$, and $g(x) = x+v$. Thus the result holds.

For general $N \in \fg$, as we observed already it suffices to prove the case when $V_{<i} = 0$ and $v_j \neq 0$ for any $j > i$ such that $m_j \neq 0$. Let $j'>i$ be the smallest integer among such $j$. Then we have an orthogonal decomposition $V=V_{>j'} \oplus V_{j'} \oplus V_i.$
Now by the argument above, if we set $G' = Sp(V_{j'} \oplus V_i)$ then there exists $g \in Z_{G'}(N)^{F'}\subset Z_{G}(N)^{F'}$ such that $g(x) = x+v_{j'}$. Therefore, if we can find $g' \in Z_G(N)^{F'}$ such that $g'(x) = x+v-v_{j'}$, then $(g\circ g')(x) = x+v$ and the claim follows. Now we use induction on $\#\{j>i \mid v_j \neq 0\}$ to complete the proof.
\end{proof}

This lemma has a following consequence. Suppose we are given $l \in \bP(\ker_i N)$, $l' \in \bP(\ker_{\geq i} N)$ such that $F'(l) = l, F'(l') = l'$, and $l\equiv l' \mod \ker_{>i} N$. (In particular, $l' \not\in \bP(\ker_{>i} N)$.) Then there exist $x, v\in V^{F'}$ such that $l=\spn{x}$ and $l'=\spn{x+v}$. Now the lemma above says that there exists $g \in Z_G(N)^{F'}$ such that $g\cdot l = l'$. It gives an $F'$-equivariant isomorphism $g: \eta_N^{-1}(l) \simeq \eta_N^{-1}(l')$. In particular, if we set $F'=\tilde{z}F$ for some $\tilde{z} \in Z_G(N)^F$, then we have
$$\tr(w\tilde{z}^*F^*, H^*(\eta_N^{-1}(l))=\tr(w\tilde{z}^*F^*, H^*(\eta_N^{-1}(l')).$$

Recall that we have (before Proposition \ref{prop:maintool})
\begin{align*}
\tr(wzF^*, H^*(\cB_N))  &= \sum_{l \in \bP(\ker N)^{\tilde{z}F}} \tr(w\tilde{z}^*F^*, H^*(\eta^{-1}_N(l))).
\\&= \sum_{i \geq 0} \sum_{l \in (\bP(\ker_{\geq i} N) - \bP(\ker_{>i} N))^{\tilde{z}F}} \tr(w\tilde{z}^*F^*, H^*(\eta_N^{-1}(l))).
\end{align*}
Now the observation above implies that 
\begin{align*}
&\sum_{l \in (\bP(\ker_{\geq i} N) - \bP(\ker_{>i} N))^{\tilde{z}F}} \tr(w\tilde{z}^*F^*, H^*(\eta_N^{-1}(l)))
\\&=
 \sum_{l \in \bP(\ker_{i} N)^{\tilde{z}F}} \#\{l' \in  (\bP(\ker_{\geq i} N) - \bP(\ker_{>i} N))^{\tilde{z}F} \mid  l'\equiv l \mod \ker_{>i}N \} \tr(w\tilde{z}^*F^*, H^*(\eta_N^{-1}(l))).
\end{align*}
In order to simplify this formula, consider a canonical projection
$$\bP(\ker_{\geq i} N) - \bP(\ker_{>i} N) \rightarrow \bP(\ker_{ i }N) $$
which is a fiber bundle with fiber isomorphic to $\ker_{>i} N \simeq \bA^{m_{>i}}$ where $m_{>i} \colonequals \sum_{j>i} m_j$. This morphism is clearly $\tilde{z}F$-equivariant, and for each $l \in \bP(\ker_{ i}N)$ there are exactly $q^{m_{>i}}$ $\tilde{z}F$-stable points in the fiber of $l$ since $\#(\bA^{d})^{F'}=q^d$ for any Frobenius $F'$ on $\bA^d$. (ref. \cite[Proposition 10.11]{dm91}) Therefore we have
\begin{align*}
\tr(wzF^*, H^*(\cB_N)) = \sum_{i \geq 0} q^{m_{>i}}  \sum_{l \in \bP(\ker_{i} N)^{\tilde{z}F}} \tr(w\tilde{z}^*F^*, H^*(\eta_N^{-1}(l)))
\end{align*}
Then for each $i$, to calculate $\sum_{l \in \bP(\ker_{i} N)^{\tilde{z}F}} \tr(w\tilde{z}^*F^*, H^*(\eta_N^{-1}(l)))$ we may apply the results in the previous section since each $V_i \subset V$ is a direct summand in the orthogonal decomposition $V=V_1\oplus V_2 \oplus \cdots$. The following theorem is now a natural consequence.

\begin{thm}\label{thm:mainq} Let $z=\prod_j z_j^{a_j} \in A(\lambda)$ for some $a_j \in \bZ$. Then $\Res^W_{W'}\tspq{\lambda,z}$ is equal to
\begin{align*}
& \sum_{i \textup{ odd}}  q^{m_{>i}}\frac{q^{m_i}-1}{q-1} \tspq{\lambda\pch{i,i}{i-1,i-1}, z}
\\& +\sum_{\substack{i \textup{ even}\\m_i \textup{ odd}}} q^{m_{>i}} \left[
\begin{aligned}
&\frac{q^{{m_i}-1}-1}{q-1}\tspq{\lambda\pch{i,i}{i-1,i-1}, z}
\\&+\frac{q^{{m_i}-1}+q^{(m_i-1)/2}}{2}\tspq{\lambda\pch{i}{i-2},z (z_iz_{i-2})^{a_i}}
\\&+\frac{q^{{m_i}-1}-q^{(m_i-1)/2}}{2}\tspq{\lambda\pch{i}{i-2},z (z_i z_{i-2})^{a_i+1}}
\end{aligned}
\right]
\\&+\sum_{\substack{i \textup{ even}\\m_i \textup{ even}}} q^{m_{>i}} \left[
\begin{aligned}
& \left(\frac{q^{m_i-1}-1}{q-1}+(-1)^{a_i}q^{m_i/2-1}\right)\tspq{\lambda\pch{i,i}{i-1,i-1}, z}
\\&+\frac{q^{{m_i}-1}-(-1)^{a_i}q^{m_i/2-1}}{2}\tspq{\lambda\pch{i}{i-2},z (z_iz_{i-2})^{a_i}}
\\&+\frac{q^{{m_i}-1}-(-1)^{a_i}q^{m_i/2-1}}{2}\tspq{\lambda\pch{i}{i-2},z (z_i z_{i-2})^{a_i+1}}
\end{aligned}
\right].
\end{align*}
\end{thm}
Thus, Theorem \ref{thm:main} is valid at least when its formula is evaluated at $x=q$. However, the proof of Theorem \ref{thm:mainq} remains valid when we replace $q$ by $q^a$ for any $a \in \bZ_{>0}$, i.e. the formula of Theorem \ref{thm:main} is valid when evaluated at infinitely many integers $\{q^a \mid a \in \bZ_{>0}\}$. As we already know that $\Res^W_{W'}\tspx{\lambda, z}$ is a polynomial in $x$, we conclude that Theorem \ref{thm:main} is true as stated.

\begin{rmk} When $G=SO_{2n+1}$ or $G=SO_{2n}$, we need to switch the condition of $i$ being even and $i$ being odd in the expression above. After this, if we start with $z \in A(\lambda)$ and after we remove terms with either zero coefficient and of the form $\tspq{\lambda\pch{1}{-1}, -}$, then one can check that there does not appear any term of the form $\tspq{\lambda', z'}$ with $z' \in \tilde{A}(\lambda') -A(\lambda')$. Therefore, Theorem \ref{thm:mainq} is also valid when $G=SO_{2n+1}$ or $G=SO_{2n}$ (after switching the condition of $i$ being even and $i$ being odd). Also, the theorem holds even when $G=SO_{2n}$ and one of $\lambda$, $\lambda\pch{i,i}{i-1,i-1}$, or $\lambda\pch{i}{i-2}$ is very even, see the remark at the end of Section \ref{sec:partial}.
\end{rmk}

\bibliographystyle{amsalphacopy}
\bibliography{betti}

\end{document}